\documentclass{llncs}
\usepackage{amssymb}
\usepackage{amsmath}
\usepackage[english]{babel}
\usepackage[ansinew]{inputenc}
\usepackage{graphicx}
\usepackage{float}
\usepackage{url}
\usepackage{color}

\newtheorem{Problem}{Problem}
\newtheorem{Theorem}{Theorem}
\newtheorem{Lemma}[Theorem]{Lemma}

\usepackage{verbatim}

\newcommand{\Tr}{\mathrm{Trd}}
\newcommand{\Nr}{\mathrm{Nrd}}
\floatstyle{ruled}
\floatname{algospec}{Algorithm}
\newfloat{algospec}{htbp}{loa}

\newcommand{\bigO}{O}

\newcommand{\QA}{B_{p,\infty}}
\newcommand{\cO}{\mathcal{O}}

\newcommand{\PGL}{\mathrm{PGL}}
\renewcommand{\AA}{\mathbb{A}}
\newcommand{\FF}{\mathbb{F}}

\newcommand{\PP}{\mathbb{P}}
\newcommand{\QQ}{\mathbb{Q}}
\newcommand{\ZZ}{\mathbb{Z}}
\newcommand{\Mat}{\mathbb{M}}
\newcommand{\End}{\mathrm{End}}
\newcommand{\disc}{\mathrm{disc}}
\newcommand{\pp}{\mathfrak{p}}

\newcommand{\dsp}{\displaystyle}
\newcommand{\kro}[2]{\left(\!\frac{#1}{#2}\!\right)}
\newcommand{\ccomma}{\raisebox{0.4ex}{,}}

\usepackage{todonotes}

\begin{document}
\title{On the quaternion $\ell$-isogeny path problem}
\author{David Kohel, Kristin Lauter, Christophe Petit\thanks{ The third author is supported by an F.R.S.-FNRS postdoctoral research fellowship at Universit\'e catholique de Louvain, Louvain-la-Neuve.}, Jean-Pierre Tignol}

\institute{Institut de Math\'ematiques de Marseille,Microsoft Research,\\ Universit\'e catholique de Louvain,Universit\'e catholique de Louvain}


\maketitle

\begin{center}
\emph{To appear in the LMS Journal of Computation and Mathematics, as a special issue for ANTS (Algorithmic Number Theory Symposium) conference.}
\end{center}

\begin{abstract}
Let $\cO$ be a maximal order in a definite quaternion algebra over $\mathbb{Q}$
of prime discriminant $p$, and $\ell$ a small prime.
We describe a probabilistic algorithm, which for a 
given left $\cO$-ideal, computes a representative in 
its left ideal class of $\ell$-power norm.
In practice the algorithm is efficient, and subject to heuristics on expected distributions of primes, 
runs in expected polynomial time. 
This breaks the underlying problem for a quaternion 
analog of the Charles-Goren-Lauter hash function, and 
has security implications for the original 
CGL construction in terms of supersingular elliptic 
curves.
\end{abstract}

\section{Introduction}
\label{sec:introduction}

In this paper, we provide a probabilistic algorithm to solve a
quaternion ideal analog of the path problem in supersingular 
$\ell$-isogeny graphs. 
The main result is an algorithm for the following.  Let $\QA$ 
be a quaternion algebra over $\QQ$ ramified at $p$ and $\infty$. Let $\ell$ be a ``small'' prime, typically 2 or 3, or any small constant prime.
Given a maximal quaternion order $\cO$ in $\QA$ and a left 
$\cO$-ideal $I$, compute an equivalent left $\cO$-ideal 
$J = I\beta$ with norm $\ell^k$ for some~$k$.
This algorithm runs in practice in probabilistic polynomial time,
and this effective runtime follows from heuristic assumptions on expected
distributions of primes.
With minimal adaptation, the algorithm also applies to output an 
ideal with smooth (or power-smooth) norm. 
The algorithm is described in terms of a special maximal order, 
but extends to any maximal order by passing through such a 
special order. 

The motivation for this problem is an explicit equivalence of 
categories between left $\cO$-ideals and supersingular elliptic 
curves (over $\bar{\FF}_p$). The Deuring correspondence gives 
a bijection between such curves, up to Galois conjugacy, and 
isomorphism classes of maximal orders in $\QA$.  This bijection 
can be turned into an equivalence of categories by the following 
construction.  Let $E_0/K$ be a fixed elliptic curve with 
endomorphism ring $\cO = \End(E_0)$ a quaternion order in 
$\QA = \cO \otimes \QQ$ (we may take the base field $K = \FF_{p^2}$ 
and $E_0$ such that $|E_0(K)| = (p+1)^2$).  Associated to any pair 
$(E_1,\varphi)$ where $\varphi: E_0 \rightarrow E_1$ is an isogeny, 
we obtain a left $\cO$-ideal $I = \mathrm{Hom}(E_1,E_0) \varphi$ 
of norm $n = \deg(\varphi)$ 
and conversely every left $\cO$-ideal arises in this way 
(see Kohel~\cite[Section 5.3]{Kohel1996}). 
In particular, given any isogeny $\psi: E_0 \rightarrow E_1$ of 
degree $m$, the left $\cO$-ideal $J = I \hat{\varphi}\psi/n$ is 
an equivalent ideal of norm $m$, where $\hat\psi$ is the dual of $\psi$.

The problem we address in this work is to solve the quaternion 
version of the supersingular $\ell$-isogeny path problem: given 
$E_0$, $E_1$ and a small prime $\ell$, find an $\ell$-power 
isogeny from $E_0$ to $E_1$.  Under this equivalence of categories, 
the analogous problem is the determination of a $\ell$-power norm
left $\cO$-ideal in the class of a given left $\cO$-ideal $I$. 
After introducing the necessary background on quaternion orders and ideals 
in Section~\ref{sec:quaternions} and addressing some preliminary algorithmic problems in Sections~\ref{sec:prel}, we solve the $\ell$-power norm problem in Section~\ref{sec:ideals:ellpow}.  Subject to reasonable heuristics on 
the probability of finding suitable primes, we obtain a probabilistic 
algorithm which solves this problem in expected polynomial time. 
The experimental runtime agrees with the most optimistic 
predictions for the distribution of primes. 

The algorithm gives a clear distinction between the efficiency 
of the $\ell$-isogeny problem in the equivalent category of 
quaternion ideals, whereas the analogous problem in the category 
of supersingular elliptic curves, on which the security of the 
Charles, Goren and Lauter hash function~\cite{Charles2009} is 
based, has to date resisted attack.  This dichotomy poses 
several questions on the extent to which the information from 
the algebraic category can be transported to the geometric one.  
In particular, one expects an algorithm for computing the 
endomorphism ring of a given elliptic curve to provide an 
effective reduction to the algebraic setting, making the 
hardness of this problem critical to the underlying security.

\section{The quaternion $\ell$-isogeny path problem}
\label{sec:quaternions}

In this section, we first motivate and define the quaternion $\ell$-isogeny path problem.
We then recall basic facts on quaternion algebras. We introduce \emph{$p$-extremal} maximal orders, which will play an important role in our solution of the quaternion $\ell$-isogeny problem. 
We finally discuss properties of reduced norms and ideal morphisms. 

\subsection{``Hard'' isogeny problems}
\label{sec:intro:hard}

The motivation for studying the quaternion $\ell$-isogeny problem 
is based on the analogous (indeed categorically equivalent) problem 
for supersingular elliptic curves.  
The difficulty of this problem for elliptic curves underlies 
the security of the Charles, Goren and Lauter hash 
function~\cite{Charles2009}.

As an example, finding a preimage (inverting the function) amounts 
to solving the following path problem in the supersingular 
$\ell$-isogeny graph:
\begin{Problem}\label{prob:preim}
Let $p$ and $\ell$ be prime numbers, $p\neq\ell$. Let $E_0$ and $E_1$ 
be two supersingular elliptic curves over $\FF_{p^2}$ with 
$|E_0(\FF_{p^2})|=|E_1(\FF_{p^2})|=(p+1)^2$. 
Find $k\in\mathbb{N}$ and an isogeny of degree $\ell^k$ from $E_0$ to $E_1$.
\end{Problem}
Similarly, finding collisions requires a solution to the following 
multiple path problem in the supersingular $\ell$-isogeny graph:
\begin{Problem}\label{prob:coll}
Let $p$ and $\ell$ be prime numbers, $p\neq\ell$. Let $E_0$ be 
a supersingular elliptic curve over $\FF_{p^2}$. 
Find $k_1,k_2\in\mathbb{N}$, a supersingular elliptic curve 
$E_1$ and two distinct isogenies (i.e.~with distinct kernels) 
of degrees respectively $\ell^{k_1}$ and $\ell^{k_2}$ from $E_0$ 
to $E_1$.
\end{Problem}

Setting $\cO = \End(E_0)$, we have a category of left $\cO$-ideals,
with morphisms $I \rightarrow I\alpha \subseteq J$, for $\alpha$ in 
$B = \cO \otimes \QQ$, which is equivalent to the category of 
supersingular elliptic curves and isogenies. 
The analog of the path problem in supersingular $\ell$-isogeny graphs
is that of finding a representative ideal $J$ for given $I$ of norm 
$\ell^k$. We call this problem the \emph{quaternion $\ell$-isogeny path 
problem}, and focus on 
its effective solution in this article.

\subsection{Quaternion algebras\label{sec:intro:QA}}

In this work we consider the structure of left ideals of a maximal order 
in the quaternion algebra $B_{p,\infty}$ ramified only at $p$ and $\infty$.  
Such an algebra is isomorphic to $\End(E) \otimes \QQ$ for any supersingular 
elliptic curve $E/\FF_{p^2}$. Here we denote $\End(E) = \End_{\bar{\FF}_p}(E)$ 
and if we assume $\#E(\FF_{p^2}) = (p+1)^2$, then the full endomorphism 
ring $\End(E)$ is defined over $\FF_{p^2}$.  
Any definite quaternion algebra over $\QQ$ has a presentation of the form 
$\QQ\langle{i,j}\rangle$, where $i^2 = a$, $j^2 = b$, $k = ij = -ji$ for 
negative integers $a,b$.  The canonical involution on $B_{p,\infty}$ is 
given by
$$
\alpha = x_0 + x_1 i + x_2 j + x_3 k \longmapsto
\bar{\alpha} = x_0 - x_1 i - x_2 j - x_3 k.
$$
from which the reduced trace and norm take the form
$$
\Tr(\alpha) = \alpha + \bar{\alpha} = 2x_0
\mbox{ and }
\Nr(\alpha) = \alpha \bar{\alpha} = x_0^2 - a x_1^2 - b x_2^2 + ab x_3^2. 
$$
The integral basis $\{1,i,j,k\}$ has the nice property of being 
an orthogonal basis with respect to the bilinear form 
$\langle{x,y}\rangle = \Nr(x + y) - \Nr(x) - \Nr(y)$
associated to the reduced norm.  
Nevertheless, the order $\cO = \ZZ\langle{i,j}\rangle$ 
is never maximal.

\subsection{Extremal orders\label{sec:prob:extremal}}

In this work we first place the focus on the {\it $p$-extremal} maximal 
orders $\cO$ containing $\pi$ such that $\pi^2 = -p$. For a general 
order there exists a unique maximal $2$-sided ideal $\mathfrak{P}$ 
over $p$, and this ideal is principal if and only if there exists such 
an element $\pi$.
The maximal ideal $\mathfrak{P}$ is a generator of the $2$-sided class 
group, and $p$-extremal orders are precisely those of trivial $2$-sided 
class number.  In the context of supersingular elliptic curves, these 
are the maximal orders which are endomorphism rings of elliptic curves 
defined over $\FF_p$ with Frobenius endomorphism~$\pi$. 

Secondly, we focus on orders with distinguished quadratic subring $R$.  
For a maximal order $\cO$ we define 
$
d(\cO) = \min\{ \disc(R) : \ZZ \ne R \subsetneq \cO \}.
$
Among all $p$-extremal maximal quaternion orders, we define 
a {\it special} $p$-extremal maximal order $\cO$ to be a 
$p$-extremal maximal order such that $d(\cO)$ is minimal.

The following lemma establishes the main properties we need for such an 
order, after which Lemmas~\ref{lem:disc-4}, \ref{lem:disc-8}, 
and~\ref{lem:disc-q} provide for their existence by explicit construction.

\begin{Lemma}
\label{lem:special-p-extremal-properties}
Let $\cO$ be a maximal order in $B_{p,\infty}$ containing a subring 
$\ZZ\langle{i,j}\rangle$ with $i^2=-q$, $j^2=-p$, and $ij = -ji$, 
for $q$ coprime to $p$. 
Set $R = \cO \cap \QQ[i]$ and let $D$ be its discriminant.
If $R$ is the ring of integers of $\QQ[i]$, then $R^\perp = Rj$ and 
$R + Rj$ is a suborder of index $|D|$ in $\cO$. 
If $\omega$ is a generator of $R$, then
$$
\Nr(x_1 + y_1\omega + (x_2 + y_2\omega)j) = f(x_1,y_1) + p f(x_2,y_2),
$$
where $f(x,y)$ is a principal quadratic form of discriminant $D$.
\end{Lemma}

\begin{proof}
The triviality of the trace of $j$ and the anti-commuting relation 
$ij = -ji$ imply that $\QQ(i)$ has orthogonal complement $\QQ(i)j$
in $B_{p,\infty}$. Consequently $R^\perp \subset \cO$ is a lattice 
in $\QQ(i)j$ containing $Rj$, hence of the form $\mathfrak{a}j$ for 
a fractional ideal $\mathfrak{a}$ of $R$ which contains $R$. 
The prime $p$ is inert in $R$, since $p$ is ramified in $B_{p,\infty}$ 
but not in $R$.  
Since the norm is integral on $\mathfrak{a}j$, and $\Nr(j) = p$, 
it follows that $\mathfrak{a}$ is integral, hence equals $R$. 
The orthogonality of $R$ and $Rj$ implies that 
$j\beta = \bar{\beta}j$ for all $\beta$ in $R$, so $jR = Rj$ 
and $R + Rj$ is closed under multiplication. 
The form of the norm follows from orthogonality and 
multiplicativity of the norm: $\Nr(\beta_1 + \beta_2j) 
= \Nr(\beta_1) + p \Nr(\beta_2)$.  Consequently the 
discriminant of the norm form is $D^2 p^2$, from which 
we conclude that $R + Rj$ has index $|D|$ in any 
maximal order. 
\end{proof}

By convention, for our special $p$-extremal order $\cO$, we 
fix $\ZZ[i] \subseteq R$ with $i^2 = -q$ and $D = \disc(R) 
= -d(\cO)$, and $j^2 = -p$ (i.e.~$j = \pi$ above). Being 
of smallest discriminant, $R$ is necessarily a maximal order 
whose discriminant is the first of the sequence 
$$
-3, -4, -7, -8, -q \text{ for prime } q \equiv 3 \bmod 4,
$$
such that $p$ is ramified or inert in $R$.
The next three lemmas establish existence for $q = 1$, $q = 2$, 
and $q \equiv 3 \bmod 4$ prime.
These lemmas incorporate and expand on Propositions~5.1 and~5.2 
of Pizer~\cite{Pizer1980}. We recall that an order in 
a quaternion algebra is {\it Eichler} if it is the intersection 
of two maximal orders. 

\begin{Lemma}
\label{lem:disc-4}
Let $p \equiv 3 \bmod 4$ be a prime, and let $B = \QQ\langle{i,j}\rangle$ 
be the quaternion algebra given by the presentation $i^2 = -1$, $j^2 = -p$,
and $k = ij = -ji$, and set $R = \ZZ[i]$. Then $B$ is ramified only at $p$ 
and $\infty$, and $\ZZ\langle{i,j}\rangle$ is contained in exactly two 
maximal orders with index $4$, described by the inclusion chains:
$$
\ZZ\langle{i,j}\rangle 
  \subsetneq \ZZ\langle{i,\frac{1+i+j+k}{2}}\rangle
  \subsetneq 
  \left\{
  \begin{array}{l}
  \dsp \ZZ\langle{i,\frac{1+j}{2}}\rangle\ccomma\\[2.5mm]
  \dsp \ZZ\langle{i,\frac{1+k}{2}}\rangle\cdot
  \end{array}
  \right.
$$
In particular $\ZZ\langle{i,(1+i+j+k)/2}\rangle$ is an Eichler order, but 
$\ZZ\langle{i,j}\rangle$ is not.
\end{Lemma}

\begin{Lemma}
\label{lem:disc-8}
Let $p \equiv 5 \bmod 8$ be a prime, and let $B = \QQ\langle{i,j}\rangle$ 
be the quaternion algebra given by the presentation $i^2 = -2$, $j^2 = -p$,
and $k = ij = -ji$, and set $R = \ZZ[i]$. Then $B$ is ramified only at $p$ 
and $\infty$, and $\ZZ\langle{i,j}\rangle$ is contained in exactly two 
maximal orders with index $8$, described by the inclusion chains:
$$
\ZZ\langle{i,j}\rangle 
  \subsetneq \ZZ\langle{i,j,\frac{i+k}{2}}\rangle
  \subsetneq \ZZ\langle{i,\frac{i+k}{2},\frac{1+j+k}{2}}\rangle
  \subsetneq 
  \left\{
  \begin{array}{l}
  \dsp \ZZ\langle{i,\frac{1+j+k}{2},\frac{i+2j+k}{4}}\rangle\ccomma\\[2.5mm]
  \dsp \ZZ\langle{i,\frac{1+j+k}{2},\frac{i+2j-k}{4}}\rangle\cdot
  \end{array}
  \right.
$$
In particular $\ZZ\langle{i,j}\rangle$ is not an Eichler order.
\end{Lemma}

\begin{Lemma}
\label{lem:disc-q}
Let $p$ and $q$ be primes, with 
$p \equiv 1 \bmod 4$,  
$q \equiv 3 \bmod 4$, and  
$$
\kro{-p}{q} = 1.
$$
Let $B = \QQ\langle{i,j}\rangle$ be the quaternion algebra given 
by the relations $i^2 = -q$, $j^2 = -p$, and $k = ij = -ji$, 
and set $R = \ZZ[(1+i)/2]$. Then $B$ is ramified only at $p$ 
and $\infty$, and $\ZZ\langle{(1+i)/2,j}\rangle = R + Rj$ is 
contained in exactly two maximal orders with index $q$, described 
by the inclusion chains:
$$
\ZZ\langle{(1+i)/2,j}\rangle 
  \subsetneq 
  \left\{
  \begin{array}{l}
  \dsp \ZZ\langle{\frac{1+i}{2}\ccomma\, j\,\ccomma\frac{ci+k}{q}}\rangle\ccomma\\[2.5mm]
  \dsp \ZZ\langle{\frac{1+i}{2}\ccomma\, j\,\ccomma\frac{ci-k}{q}}\rangle\ccomma
  \end{array}
  \right.
$$
where $c$ is any root of $x^2 + p \bmod q$. In particular $R + Rj$ 
is an Eichler order.
\end{Lemma}

Under the generalized Riemann hypothesis, for $p \equiv 1 \bmod 4$, 
the smallest $q$ satisfying the conditions of the last lemma is 
$O(\log(p)^2)$ by a result of Ankeny~\cite{Ankeny1952} (or explicitly 
$q < 2\log(p)^2$ by Bach~\cite{Bach1990}).
In the remainder of this paper, we will assume that $\QA$, $\cO$, 
and $R$ are suitably constructed from these 
lemmas with $\disc(R)$ the minimal discriminant in which $p$ is inert 
in the sequence $-3$, $-4$, $-7$, $-8$, or $-q$ for $q \equiv 3 \bmod 4$ 
prime. 

\subsection{Reduced norms and ideal morphisms}
\label{sec:ideals:properties}

Now suppose that $\cO$ is any maximal order.
We recall that the reduced norm on $\QA$ induces a reduced norm 
on left ideals defined by any of the equivalent conditions
$$
\Nr(I) :=  \sqrt{|\cO/I|} 
= \gcd\left(\{\, \Nr(\alpha) \;:\; \alpha \in I \,\}\right),
$$
or by $I\bar{I} = \Nr(I)\cO$.  It follows that the 
reduced norm on ideals is multiplicative 
and compatible with the reduced norm on elements $\Nr(\alpha) = 
\Nr(\alpha\cO) = \Nr(\cO\alpha)$. 
If $I$ and $J$ are left $\cO$-ideals, a homomorphism of $I$ to $J$ 
is a map given by $\alpha \mapsto \alpha\gamma$ for $\gamma$ in 
$\QA^*$, which is an isomorphism if $J = I\gamma$.
By the multiplicativity of the reduced norm, isomorphisms 
are similitudes of quadratic modules (with respect to the 
reduced norm).  In particular, an isomorphism sends a 
reduced basis to a reduced basis. In fact the normalized 
norm map 
$$
q_I = \frac{\Nr}{\Nr(I)} : I \longrightarrow \ZZ
$$
remains invariant under this isomorphism, in the sense that $q_I(\alpha)
= q_J(\beta)$ for $\alpha$ in $I$ and $\beta = \alpha\gamma$ in $J$.
The normalized norm $q_I$ is a positive-definite integral quadratic map, 
whose bilinear module given by 
$
\langle{x,y}\rangle = q_I(x+y) - q_I(x) - q_I(y)
$
has determinant $p^2$. This follows from the same property for any maximal order 
(see Pizer~\cite[Proposition~1.1]{Pizer1980}), since $|\cO/I| = \Nr(I)^2$,
and the fact that any submodule of index $m$ in a quadratic module $L$ 
has determinant $m^2\det(L)$.

The following lemma serves to replace an ideal $I$ with an isomorphic 
one of different reduced norm.

\begin{Lemma}
\label{lem:ideal_norm_rep}
Let $I$ be a left $\cO$-ideal of reduced norm $N$ and $\alpha$ 
an element of $I$. 
Then $I\gamma$, where $\gamma = \bar{\alpha}/N$, is a left 
$\cO$-ideal of norm $q_I(\alpha)$.
\end{Lemma}

\begin{proof}
By the multiplicativity of the reduced norm, and $\Nr(\alpha) = \Nr(\bar{\alpha})$, 
we have 
$$
\Nr(I\gamma) = \Nr(I)\Nr(\gamma) 
= N\frac{\Nr(\alpha)}{N^2} = \frac{\Nr(\alpha)}{N} = q_I(\alpha).
$$
Clearly $I$ is a fractional left $\cO$-ideal, so it remains 
to show that $I\gamma \subseteq \cO$.  
Since $\cO\alpha \subseteq I$, we have $\bar{\alpha} \subseteq \bar{I}$, 
and hence $I\bar{\alpha} \subseteq I\bar{I} = N\cO$, from which 
$I\gamma \subseteq \cO$ follows.
\end{proof}

\section{Preliminary algorithmic results\label{sec:prel}}

In this section, we provide two algorithmic tools that will be used to solve the quaternion $\ell$-isogeny path problem in Section~\ref{sec:ideals:ellpow}. The first algorithm computes prime norm representatives in ideal classes. The second one computes representations of integers by the norm form of a $p$-extremal order.


\subsection{Computing prime norm representatives in ideal classes}
\label{sec:ideals:prime}


Given a maximal order $\cO$ and a left $\cO$-ideal $I$, we 
give a probabilistic 
algorithm that computes another left $\cO$-ideal $J = I\gamma$ 
in the same class, but with prime norm.
Using Lemma~\ref{lem:ideal_norm_rep}, this problem reduces to the  problem of finding a prime represented by $q_I$. 
\vspace{2mm}

\noindent{\bf Prime norm algorithm.}
Given a left $\cO$-ideal $I$ of norm $N$, with a Minkowski-reduced basis 
$\{\alpha_1,\alpha_2,\alpha_3,\alpha_4\}$. 
Generate random elements $\alpha = \sum_i x_i \alpha_i$ with 
$(x_1,x_2,x_3,x_4)$ in a box $[-m,m]^4$ until finding an element $\alpha$ 
of $I$ with $q_I(\alpha)$ prime, and return $I(\bar{\alpha}/N)$.
\vspace{2mm}

Assuming that numbers represented by $q_I$ behave like random 
numbers, it remains to ensure that $q_I([-m,m]^4)$ contains 
sufficiently many primes to have a high probability of finding 
one.
If $\{\alpha_1,\alpha_2,\alpha_3,\alpha_4\}$ is a Minkowski-reduced 
basis, the $q_I(\alpha_i)$ attain the successive minima, and we 
have the bounds
$$
p^2 \le 16 q_I(\alpha_1) q_I(\alpha_2) q_I(\alpha_3) q_I(\alpha_4) \le 4 p^2,
$$
where $q_I(\alpha_i) \le q_I(\alpha_{i+1})$. For a generic ideal 
$I$ we expect $q_I(\alpha_4)$ to be in $\tilde{\bigO}(\sqrt{p})$.  
In the worst case, $q_I(\alpha_4)$ is in $\tilde{\bigO}(p)$ when 
$I$ equals an order $\cO$ containing a subring $R$ with $|\disc(R)|$ 
in $O(\log(p)^n)$.
Assuming $I$ is generic, we expect to find $\alpha$ with 
$q_I(\alpha)$ in $\tilde{\bigO}(m^2 \sqrt{p})$.
In practice, we find sufficiently many primes $q_I(\alpha)$ for 
$m$ which grows polynomially in $\log(p)$. 
However to provably terminate, even under the GRH, it may be 
necessary to allow $m$ to exceed a function in $O(\sqrt[4]{p})$, in which case the output may exceed $O(p)$.

We implemented a prime norm algorithm in Magma~\cite{Magma}. 
We tested it on ideals of $\ell$-power norms generated via 
a random walk from a given maximal order. 
All our computations with primes of up to 200 bits and 
random ideals took seconds on an Intel Xeon CPU X5500 
processor with 24 GB RAM running at 2.67GHz. 
The norms of the output ideals $J$ were experimentally only 
slightly larger than $\sqrt{p}$. 
The experimental results are given in Appendix~\ref{sec:primeres}.

\subsection{Representing integers by special orders}
\label{sec:repns_in_orders}

We also consider the problem of representing a sufficiently 
large positive integer $M$ by the norm form of $\cO$.  Suppose that 
$\cO$ is a $p$-extremal order, with suborder $R + Rj$, and 
let $D = \disc(R)$.  We let $\Phi(x)$ be a monotone function such 
that a suitable interval $[x,x+\Phi(x)]$ contains
sufficiently many primes,
and we assume that $M \ge p\,\Phi(M)$.  
If $\omega$ is a reduced generator of $R$ (of trace $0$ or $\pm 1$), 
then the norm form on $R + Rj$ is of the form
$$
\Nr(\alpha + \beta j) = f(x_1,y_1) + p f(x_2,y_2),
$$
where $\alpha = x_1 + y_1 \omega$ and $\beta = x_2 + y_2 \omega$, and 
$f(x,y)$ is a principal form. For $(x,y)$ in $[-m,m]^2$ with 
$m = \lfloor\sqrt{\Phi(M)/|D|}\rfloor$, we have $f(x,y) < \Phi(M)$ 
and $\Nr(\beta j) < p\,\Phi(M) < M$. 
This gives the following algorithm on which we build our strong 
approximation algorithm.
\vspace{2mm} 

\noindent{\bf Integer representation.}
Given an integer $M \ge p\,\Phi(M)$. 
Set $m = \lfloor\sqrt{\Phi(M)/|D|}\rfloor$, and choose 
$(x_2,y_2)$ at random in $[-m,m]^2$ until finding a prime 
$r = M - p f(x_2,y_2)$ which is split in $R$ and for which 
a prime $\mathfrak{r}$ over $r$ is principal.  
Let $\alpha = x_1 + y_1 \omega$ be a generator for $\mathfrak{r}$,
set $\beta = x_2 + y_2 \omega$, and return $\alpha + \beta j$.
\vspace{2mm} 

Clearly the output has norm $M$. We assume that primes have density 
$1/\log(M)$ in the arithmetic progression $M - p\,[0,\Phi(M)]$. 
Moreover we assume that such primes are equidistributed among 
primes which are non-split and split in $R$ and, in the latter 
case, among each of the $h(R)$ ideal classes of $R$. 
Finally, we must assume that elements $\beta = x_2 + y_2 \omega$ 
give rise to integers $r = M - p\,\Nr(\beta)$ with the same primality probabilities as random integers in the range
$M - p\,[0,\Phi(M)]$.  
Under such heuristic assumptions, the expected number of random 
$\beta$ to be tested is $2 h(R) \log(M)$. 
Detecting a prime~$r$, solving for a representative prime $\mathfrak{r}$ 
over $r$, and determination of a principal generator can be done 
in expected polynomial time by Cornaccia's algorithm~\cite{Cornacchia1903}.

Under the heuristic assumptions made above, we can appeal to average 
distributions among all arithmetic progressions $a - p\,[0,\Phi(M)]$,
for representatives $a$ of $(\ZZ/p\ZZ)^*$. In the application that 
follows, $M$ will be of the form $\ell^e$ or $N\ell^e$, and we can 
adapt to failure to find primes in a particular arithmetic progression 
sparsely populated with primes by changing~$e$.

\section{Main algorithm}
\label{sec:ideals:ellpow}

%

In this section, we provide an algorithm to solve the quaternion $\ell$-isogeny path problem. We also sketch a generalization of our approach to build ideal class representatives with powersmooth norms.

\subsection{Overview of the algorithm}
\label{sec:algorithm_overview}


We reduce the quaternion $\ell$-isogeny problem to a restricted version of the same problem, where we assume that $\cO$ is a special $p$-extremal maximal order 
with suborder $R + Rj$  as defined in Section~\ref{sec:intro:QA}. We also assume that $I$ is a left $\cO$-ideal with reduced norm $N$, where $N$ is a (large) prime coprime to $\ell$, $|\disc(R)|$ and $p$. 
 A reduction from generic left $\cO$-ideals to left $\cO$-ideals with the required norms can be effectively performed with the algorithm of Section~\ref{sec:ideals:prime}.
A reduction from general maximal orders to special $p$-extremal orders will be provided in Section~\ref{sec:ideals:gen}.

Using Lemma~\ref{lem:ideal_norm_rep}, the quaternion $\ell$-isogeny path problem is also reduced to an effective strong approximation theorem in Section~\ref{sec:ideals:strongapproximation}. In particular if the ideal is given by a pair of generators $I = \cO(N,\alpha)$, the quaternion $\ell$-isogeny path problem is reduced  to finding $\lambda \in \ZZ$ coprime to $N$ and 
$$
\beta \equiv \lambda\alpha \bmod N\cO
$$ 
with $\Nr(\beta) = N\ell^e$ for some positive integer~$e$.

Sections~\ref{sec:ideals:ellpow:isom}, \ref{sec:ideals:ellpow:lift}, and~\ref{sec:ideals:ellpow:results} describe the core of our approach to solve this problem.
Since the index of $R + Rj$ in $\cO$ is coprime to $N$, we 
have an isomorphism
$$
\frac{R + Rj}{N(R + Rj)} \cong \frac{\cO}{N\cO}\cdot
$$
We can therefore choose representative elements in $R + Rj$ 
as convenient to simplify the algorithm. Since the index 
$[\cO:R + Rj] = |\disc(R)|$ is assumed to be small (in $O(\log(p)^2)$ under the GRH), the size of the output might be slightly 
larger, but the distinction is asymptotically insignificant.
%
A direct approach to the strong approximation problem 
to solve for $\beta$ seems daunting, so instead we reduce 
to the following steps:
\begin{enumerate}
\item
Solve for a random $\gamma \in \cO$ of reduced norm $N\ell^{e_0}$.
\item
Solve for $[\mu]$ in $(\cO/N\cO)^*$ such that $(\cO\gamma/N\cO)[\mu] = I/N\cO$.
\item
Solve for the strong approximation of $[\mu]$ (modulo $N$) 
by $\mu$ in $\cO$ of reduced norm $\ell^{e_1}$.
\end{enumerate}
Here we denote the element $\mu + N\cO$ of $\cO/N\cO$ by $[\mu]$ 
to distinguish it from the conjugate $\bar{\mu}$ of $\mu $. 
The output $\beta = \gamma\mu$ is then an element of $I$ with 
reduced norm $N\ell^e$ where $e = e_0+e_1$.
The element $\gamma$ can be constructed with the algorithm of 
Section~\ref{sec:repns_in_orders}. We solve 
for $[\mu]$ by linear algebra in Section~\ref{sec:ideals:ellpow:isom}, showing that we can take 
$[\mu]$ in $(R/NR)^*[j] \subseteq (\cO/N\cO)^*$. 
The core of the algorithm is the final specialized strong 
approximation algorithm of Section~\ref{sec:ideals:ellpow:lift}, 
taking $[\mu]$ in $(R/NR)^*[j]$ and constructing the lifting 
$\mu$ of norm $\ell^e$. 
The whole algorithm for $p$-extremal orders is analyzed in Section~\ref{sec:ideals:ellpow:results}. 

As mentioned above, we finally remove the $p$-extremal condition in Section~\ref{sec:ideals:gen} by providing a reduction from the general case to the case of $p$-extremal orders, and we generalize our approach to compute ideal representatives of smooth or powersmooth norms in Section~\ref{sec:ideals:psmooth}.



\subsection{Effective strong approximation\label{sec:ideals:strongapproximation}}

Let $B := B_{p,\infty}$ be the quaternion algebra ramified at $p$ and $\infty$.
Let $\AA_\QQ$ be the rational ad\`ele ring, defined as the restricted 
product of $\QQ_v$ with respect to $\ZZ_v$, let $\ell \ne p$ be a 
``small'' prime, and let $\AA_{\QQ,\ell}$ be the restricted product 
over all $v \ne \ell$.  
Let $\AA_B = B \otimes_\QQ \AA_\QQ$ be the ad\`ele ring of $B$, 
and $\AA_{B,\ell} = B \otimes \AA_{\QQ,\ell}$.   
Then $B$ embeds diagonally in $\AA_B$ and is discrete in 
$\AA_B$ (see~\cite[Section 14]{Cassels1967}). The strong 
approximation theorem (see~\cite[Section 15]{Cassels1967}) 
asserts that $B$ is dense in $\AA_{B,\ell}$ (see also 
Th\'eor\`eme Fondamental 1.4, p.~61 of Vign\'eras~\cite{Vigneras1980}).

The strong approximation theorem can be viewed as a strong version 
of the Chinese remainder theorem.  We apply this to find an element 
of a left $\cO$-ideal $I$ which generates $I$ almost everywhere.  
Each such ideal is known to be generated by two elements $N$ and 
$\alpha$, where we may take $N = \Nr(I)$ for the first generator.
This follows since locally $\cO_v = \cO \otimes \ZZ_v$ is a left 
principal ideal ring, hence so is the quotient $\cO/N\cO$. 

If $I = \cO(N,\alpha):=\cO N+\cO\alpha$, the approximation theorem 
implies that we can find $\beta$ in $I$ such that 
$$
\beta \equiv \alpha \bmod N\cO
$$ 
and $\Nr(\beta) = N\ell^e$ for some positive integer~$e$,
from which $I = \cO(N,\alpha) = \cO(N,\beta)$.  
By Lemma~\ref{lem:ideal_norm_rep}, an effective version of this strong approximation theorem is sufficient to solve the quaternion $\ell$-isogeny path problem.
In particular, since $\beta$ is in $I$, the ideal $I\bar\beta/N$ is an isomorphic ideal of norm $\ell^e$.  

Similarly, solving for
$$
\beta \equiv \lambda\alpha \bmod N\cO
$$ 
with $\lambda \in \ZZ$ coprime to $N$ such that we still have 
$I = \cO(N,\beta)$, is also sufficient to solve the quaternion $\ell$-isogeny path problem.
We will focus on this relaxed effective strong approximation theorem in the next subsections.

\subsection{Isomorphism of $\cO/N\cO$-ideals}
\label{sec:ideals:ellpow:isom}

In this section, let $I$ be a left $\cO$-ideal of prime norm $N \ne p$,
and let $\gamma$ be an arbitrary element of $\cO$ of norm $NM$, where 
$\gcd(N,M) = 1$.  
Since $N$ is large, we can assume that it does not divide the 
index $[\cO:R+Rj]$, hence we have equalities of rings
$$
\cO/N\cO = (R+Rj)/N(R + Rj) \cong \Mat_2(\ZZ/N\ZZ). 
$$ 
We denote by $[\alpha]$ the class of an element $\alpha$ in $\cO/N\cO$ 
(as distinct from its conjugate $\bar{\alpha}$).

We note that $\cO\gamma/N\cO$ and $I/N\cO$ are proper nonzero 
left $\cO/N\cO$-ideals.  The following explicit classification 
of such ideals, in $\Mat_2(\ZZ/N\ZZ)$, will let us construct 
an explicit isomorphism between these ideals.

\begin{Lemma}
Let $N$ be a prime and $A = \Mat_2(\ZZ/N\ZZ)$. There exists a bijection
$$
S : \PP^1(\ZZ/N\ZZ) \times \PP^1(\ZZ/N\ZZ) \longrightarrow 
\frac{\{\, \gamma \in A\backslash\{0\} \;:\; \det(\gamma) = 0 \,\}}{(\ZZ/N\ZZ)^*}\ccomma 
$$
given by 
$$
S\big((u:v),(x:y)\big) = 
\left(\begin{array}{@{}cc@{}} ux & uy \\ vx & vy \end{array}\right)\cdot
$$
Under this correspondence, the set of proper nontrivial left $A$-ideals 
is in bijection with the set
$$
\{\, \PP^1(\ZZ/N\ZZ) \times (x:y) : (x:y) \in \PP^1(\ZZ/N\ZZ) \,\},
$$
and the right action of $A^*/(\ZZ/N\ZZ)^* = \PGL_2(\ZZ/N\ZZ)$ on 
left $A$-ideals is transitive and induced by the natural (transpose)
action on $\PP^1(\ZZ/N\ZZ)$.
\end{Lemma}

\begin{proof}
The nonzero matrices 
of determinant zero, modulo $(\ZZ/N\ZZ)^*$, determine a hypersurface 
$ad = bc$, which is the image of $\PP^1 \times \PP^1$ by the 
Segre embedding in $\PP^3$ (= $(A\backslash\{0\})/(\ZZ/N\ZZ)^*$). 
It is easily verified that left and right multiplication induce
the standard and transpose multiplication on the first and second 
factors of $\PP^1 \times \PP^1$, respectively, under this isomorphism, 
from which the result follows.
\end{proof}

Using an explicit isomorphism $\cO/N\cO \cong \Mat_2(\ZZ/N\ZZ)$, 
by this lemma we can find $[\mu]$ in $(\cO/N\cO)^*$ such that 
$
(\cO\gamma/N\cO) [\mu] = I/N\cO,
$
using linear algebra over $\ZZ/N\ZZ$.

In Section~\ref{sec:ideals:ellpow:lift} we require an input $[\mu]$ 
which is a unit in $Rj/N\cO$. 
Observing that $[j]$ is a unit, we see that such units form 
a coset of $(R/NR)^*$:
$$
(\cO/N\cO)^* \cap Rj/N\cO = (R/NR)^*[j].
$$
We note that $(R/NR)^*$ acts on the $N+1$ proper nontrivial left 
$\cO$-ideals, with kernel $(\ZZ/N\ZZ)^*$.
By hypothesis, $R$ is a subring of small discriminant in which 
$N$ is not ramified.  
If $N$ is inert in $R$, then the $N+1$ ideals form one orbit. 
Otherwise, if $N$ is split, there is one orbit of size $N-1$ and 
two fixed points $\cO\pp_1/N\cO$ and $\cO\pp_2/N\cO$, where 
$\pp_1$ and $\pp_2$ are the prime ideals of $R$ over $N$.
With overwhelming probability, $I/N\cO$ and $\cO\gamma/N\cO$ will 
not be such fixed points, and so we can solve for $[\mu]$ in 
$(R/NR)^*[j]$. In the event of failure, we can select a new 
$\gamma$ or $N$.

\subsection{Approximating elements of $(R/NR)^*[j]$ by $\ell$-power norm representatives}
\label{sec:ideals:ellpow:lift}

In this section, we assume that $\ell$ is a quadratic non-residue modulo $N$.
Let also $\omega$ be a generator of $R$ of minimal norm, either $1$, $2$, 
or $(1+q)/4$, for $q$ a prime congruent to $3$ modulo~$4$.
We now motivate the restriction to elements of $(R/NR)^*[j]$ 
in the previous section.

We suppose that we are given as input a lift 
$\mu_0 = x_0 + y_0 \omega + (z_0 + w_0 \omega) j$
of an arbitrary element of $\cO/N\cO$ to $R + Rj$.  
The relaxed approximation problem is to search for $\lambda$ in $\ZZ$ 
and $\mu_1 = x_1 + y_1 \omega + (z_1 + w_1 \omega) j$ such that 
$\mu = \lambda\mu_0 + N\mu_1$ satisfies the norm equation 
$$
\Nr(\mu) = 
    f(\lambda x_0 + N x_1, \lambda y_0 + N y_1) + 
p\, f(\lambda z_0 + N z_1, \lambda w_0 + N w_1) = \ell^e,
$$
for some $e\in\mathbb{N}$, where $f(x,y) = \Nr(x + y\omega)$ 
is a principal binary quadratic form of discriminant $D$ as 
in Lemma~\ref{lem:special-p-extremal-properties}.
The key idea to solve this norm equation, as used in~\cite{Petit2008c} 
to cryptanalyze the other hash function of Charles-Goren-Lauter, 
is that it simplifies considerably when $x_0 = y_0 = 0$: 
\begin{equation}
\label{eq:simplified-normequation}
\Nr(\mu) = N^2 f(x_1, y_1) + p\, f(\lambda z_0 + N z_1, \lambda w_0 + N w_1) = \ell^e.
\end{equation}
The simple algorithm we now describe to solve this equation justifies 
the choice of $[\mu] \in (R/NR)^*[j]$ in Section~\ref{sec:ideals:ellpow:isom}.


To construct $\mu$, given $[\mu] \in (R/NR)^*[j]$, we consider a first 
lift $\mu_0 = (z_0 + w_0 \omega) j$ to $Rj$ as above, and find 
$\lambda$ in $\ZZ$ and $\mu_1  = (x_1 + y_1 \omega) + (z_1 + w_1 \omega) j$ 
in $R + Rj$ satisfying the simplified 
equation~\eqref{eq:simplified-normequation}.
This equation modulo $N$, gives
$
\lambda^2 p\,f(z_0,w_0) = \ell^e \bmod N,
$
and since $\ell$ is a quadratic nonresidue modulo $N$, 
we choose the parity of $e$ depending on whether $p\,f(z_0,w_0)$ 
is a quadratic residue modulo $N$ or not, and solve for a 
square root modulo $N$ to find $\lambda$, in $0 < \lambda < N$.

Now for fixed $z_0$, $w_0$, and $\lambda$, 
Equation~\eqref{eq:simplified-normequation} implies 
a linear equation in $z_1$ and $w_1$:
\begin{equation}
\label{eq:linear-normequation}
2\lambda p L((z_0,w_0),(z_1,w_1)) = \frac{\ell^e-\lambda^2 p f(z_0,w_0)}{N}\bmod N,
\end{equation}
where $L$ is the bilinear polynomial
$$
L((z_0,w_0),(z_1,w_1)) = \langle{z_0 + w_0\omega,z_1 + w_1\omega}\rangle
                   = 2 z_0z_1 + \Tr(\omega)(z_0w_1 + w_0z_1) + 2\Nr(\omega)w_0w_1.
$$
%
Since $N$ is a large prime, such that $\gcd(x_0w_0|D|p,N)=1$, 
there are exactly $N$ solutions $(z_1,w_1)$ to the linear 
equation~\eqref{eq:linear-normequation}. 
We choose a random solution satisfying 
$$
|\lambda z_0 + N z_1| < N^2 \text{ and } 
|\lambda w_0 + N w_1| < N^2,  
$$
and equation~\eqref{eq:simplified-normequation} now leads to a problem
of representation of an integer by a binary quadratic form:
\begin{equation}
\label{eq:cornaccia-normequation}
f(x_1,y_1) = r := \frac{\ell^e- p f(\lambda z_0 + N z_1, \lambda w_0 + N w_1)}{N^2}\cdot
\end{equation}
We assume that $e$ was chosen sufficiently large so that $r$ is positive.  
If $r$ (or $rq$), modulo a smooth square integer factor, is prime, splits 
and is a norm in $R$, Cornaccia's algorithm~\cite{Cornacchia1903} can 
efficiently solve this equation, or determine that no solution exists.  
In the latter case, we repeat with a new value of $(z_1,w_1)$.
Assuming the values of $r$ behave as random values around $N^4|D|p$, 
we expect to choose $\log(N^4|D|p)h(D)$ values before finding a solution.

In practice, we begin with $e$ the minimal possible value having 
the correct parity, then we progressively increase it if no 
solution has been found. For $N$ in the range $\tilde{O}(\sqrt{p})$,
we expect the size of $e$ to satisfy 
$e \sim \log_\ell(N^4|D|p) \sim 3\log_\ell(p)$.

\subsection{Algorithm analysis and experimental results}
\label{sec:ideals:ellpow:results}

We summarize our algorithm to compute an $\ell$-power norm 
representative of a left $\cO$-ideal, where $\cO$ is a special 
$p$-extremal maximal order.

\begin{Theorem}
Let $\cO$ be a maximal order in a quaternion algebra $B_{p,\infty}$ 
and let $\ell$ be a small prime. There exists a probabilistic algorithm, 
which takes as input a left $\cO$-ideal and outputs an isomorphic 
left $\cO$-ideal of $\ell$-power reduced norm. 
\end{Theorem}

Under the most optimistic heuristic assumptions on randomness of 
representations of integers by quadratic forms and uniform 
distributions of primes, this algorithm is expected to run in 
polynomial time and to produce ideals of norm $\ell^{e}$, where
$$
e \sim \log_\ell(N p\,\Phi(p) |D|) + \log_\ell(N^4 |D| p) - \log_\ell N^2,
$$
where the three terms respectively account for the norms 
of $\gamma$, $\mu$ and $N^{-1}$.
Assuming that $\log_\ell(N) \sim \frac{1}{2}\log_\ell(p)$ 
and that in practice $\Phi(p)\sim\log(p)^n$ suffices, 
this leads to 
$$
e \sim \frac{7}{2}\log_\ell(p).
$$

We implemented the algorithms of this article in Magma~\cite{Magma}. 
We first tested the algorithm of Section~\ref{sec:repns_in_orders} to compute 
$N$ times $\ell$-power norm elements in $\cO$ with $\ell\in\{2,3\}$, for random 
primes $p$ of sizes up to 200 bits and  for $N$ values obtained after applying 
the algorithm of Section~\ref{sec:ideals:prime} on an ideal generated via a 
random walk from $\cO$. The norm of the outputs were close to the expected 
values.

We then tested the algorithm of Section~\ref{sec:ideals:ellpow:lift}
for $\ell\in\{2,3\}$, for random $p$ values of sizes up to 200 bits, 
for $N$ values obtained after applying the algorithm of 
Section~\ref{sec:ideals:prime} on an ideal generated via a random walk 
from $\cO$, and for $\mu_0 = (z_0+w_0\omega)j$ with randomly chosen 
$z_0,w_0\in\ZZ/N\ZZ$ not both equal to zero.
The exponents of the norms of the quaternions computed were 
close to the expected value $3\log_\ell p$.

We finally tested the overall algorithm of Section~\ref{sec:ideals:ellpow} 
for $\ell\in\{2,3\}$, for random $p$ values of sizes up to 200 bits, and for 
ideals $I$ generated via a random walk from $\cO$.
The $\ell$-valuation of the norm of the ideals computed were close to the 
expected value $\frac{7}{2}\log_\ell p$.

All computations were carried out on an Intel Xeon CPU X5500 processor 
with 24 GB RAM running at 2.67GHz. The algorithm of 
Section~\ref{sec:ideals:ellpow:lift} succeeded in less than 100 seconds 
for all 200 bit primes, and the overall algorithm of 
Section~\ref{sec:ideals:ellpow} terminated in less than 250 seconds 
for primes in this range. Additional experimental results are provided 
in Appendix~\ref{sec:expres}.

\subsection{Generalization to arbitrary orders}
\label{sec:ideals:gen}

We now describe how to remove the condition that $\cO$ is one of 
the special orders defined in Section~\ref{sec:intro:QA}. 
First we encode the relation between two maximal orders embedded 
in $\QA$ in terms of an associated ideal.

\begin{Lemma}
\label{lem:eichler-order-ideals}
Suppose that $\cO_1$ and $\cO_2$ are given maximal orders 
in $\QA$.  Then the Eichler order $\cO_1 \cap \cO_2$ has the same 
index in each of $\cO_1$ and $\cO_2$, which we denote $M$, and the 
set:
$$
I(\cO_1,\cO_2) = \{ \alpha \in \QA \;|\; \alpha \cO_2 \bar{\alpha} \subseteq M\cO_1 \}
$$
is a left $\cO_1$-ideal and right $\cO_2$-ideal of reduced norm $M$. 
Conversely, if $I$ is a left $\cO_1$-ideal with right order $\cO_2$, 
such that $I \not\subseteq n\cO_1$ for any $n > 1$, then $I = I(\cO_1,\cO_2)$.
\end{Lemma}

\begin{proof}
The determinant of the norm form of any maximal order $\cO$ is $p^2$, 
and for any sub-lattice $L \subset \cO$ of index $M$, the reduced norm 
form on $L$ has determinant $M^2\det(\cO)$. This establishes the 
well-known result that the index of an Eichler order in any maximal 
order is an invariant, called its level.

It is clear by construction that $I(\cO_1,\cO_2)$ is a left 
$\cO_1$-module and a right $\cO_2$-module. 
Locally at any prime $q$, we may assume $\cO_1$ and $\cO_2$ 
are $\ZZ_q$-orders such that $\cO_2 = \alpha^{-1} \cO_1 \alpha$, 
for some $\alpha$ in $\cO_1$ hence also in $\cO_2$.  
It follows that we have an inclusion $\alpha \cO_2 = \cO_1 \alpha 
\subseteq I(\cO_1,\cO_2)$. 
However, removing any integer factors (in the center), the reduced 
norm of a minimal $\alpha$ must equal the level $M\ZZ_q$, which 
implies equality. The global result follows from the local-global 
principle. 

Conversely, since any left $\cO_1$-ideal $I$ is locally principal 
at each prime $q$, 
one can find locally $\alpha$ such that $I = \cO_1\alpha$; the right order of I is then
$\cO_2 = \alpha^{-1} \cO_1 \alpha$.
By hypothesis $\alpha$ is not divisible by any integer and we 
conclude that the Eichler order has level $\Nr(\alpha) = \Nr(I) 
= M\ZZ_q$. From the above construction in terms of a local 
generator, we conclude $I = I(\cO_1,\cO_2)$.
\end{proof}

\begin{Theorem}
Let $\cO_1$ and $\cO_2$ be maximal orders in a quaternion 
algebra $B_{p,\infty}$ and let $\ell$ be a small prime. 
Given an algorithm which takes as input a left $\cO_1$-ideal 
and outputs an equivalent left $\cO_1$-ideal of $\ell$-power 
reduced norm, there exists an algorithm with the same 
complexity, up to a constant of size polynomial in the input 
size of $\cO_1$ and $\cO_2$, which takes as input a left 
$\cO_2$-ideal and outputs an equivalent left $\cO_2$-ideal of $\ell$-power reduced norm. 
\end{Theorem}

\begin{proof}
Assume we are given two orders $\cO_1$, $\cO_2$ and a left 
$\cO_2$-ideal $J$, and set $I = I(\cO_1,\cO_2)$ as in 
Lemma~\ref{lem:eichler-order-ideals}. 
The ideal $I$ may be of arbitrarily large norm, but is bounded 
by something polynomial in the specification of $\cO_1$ and 
$\cO_2$ in terms of a basis for $\QA$. 

Supposing that we have an algorithm for $\cO_1$, we find 
representative left $\cO_1$-ideals for $I$ and $IJ$ such that  
$I_1 = I\bar{\gamma}_1/\Nr(I)$ with $\gamma_1$ in $I$, and 
$I_2 = IJ\bar{\gamma}_2/\Nr(IJ)$ with $\gamma_2$ in $IJ$, where 
$$
\Nr(\gamma_1) = \Nr(I) \ell^{e_1} \mbox{ and }
\Nr(\gamma_2) = \Nr(IJ) \ell^{e_2}.
$$
It follows that $\gamma = \bar{\gamma}_1 \gamma_2/\Nr(I)$
is an element of $J$ with reduced norm $\Nr(\gamma) = 
\Nr(J) \ell^{e_1+e_2}$, and hence $J\bar{\gamma}/\Nr(J)$ 
is of reduced norm $\ell^{e_1+e_2}$.
\end{proof}

This provides a reduction of the general case to the case of
special $p$-extremal orders, at the cost of two applications 
of the algorithm of Section~\ref{sec:ideals:ellpow}, and 
a larger power of $\ell$. 

\subsection{Generalization to powersmooth norms}
\label{sec:ideals:psmooth}

We recall that a number $s=\prod\ell_i^{e_i}$ is $S$-powersmooth if $\ell_i^{e_i}<S$. 
Our algorithms can be easily modified to construct ideal representatives of powersmooth 
norms.  
Using the approximations as before, the norm should be of size close to $p^{7/2}$. 
Since the product of all maximal powers of a prime lower than $S$ can be approximated 
by $S^{S/\log S}$, an adaptation of our algorithms will allow us to compute 
$S$-powersmooth representatives of left ideal classes of $\cO$, with 
$S\approx\frac{7}{2}\log p$.

\section{Conclusion and future work}
\label{sec:concl}

In this paper, we provided a probabilistic algorithm to solve a
quaternion ideal analog of the path problem in supersingular 
$\ell$-isogeny graphs. The algorithm runs in expected polynomial 
time subject to heuristics on expected distributions of primes,
and it is efficient in practice.

Following Deuring~\cite{Deuring1941}, there is a one-to-one correspondence 
between supersingular elliptic curves modulo $p$, up to Galois conjugacy, 
and isomorphism classes of maximal orders in the quaternion algebra $\QA$. 
By identifying isogeny kernels with powersmooth ideals in the quaternion 
algebra graphs, we expect our techniques to lead to both partial attacks 
on Charles-Goren-Lauter's isogeny based hash function (when the initial 
curve has extremal endomorphism ring), and to security reductions to the 
problem of computing the endomorphism ring of a supersingular elliptic curve. 
Similarly, we expect our results to lead to a constructive version of 
Deuring's correspondence from maximal orders in $\QA$ to their corresponding 
elements in the category of supersingular elliptic curves.

\paragraph{Acknowledgements} 
The research leading to these results has received funding from the Fonds National de la Recherche - FNRS and from the European Research Council through the European ISEC action HOME/2010/ISEC/AG/INT-011 B-CCENTRE project.

\appendix

\section{Experimental results}
\label{sec:expres}

In our experiments, the value of $m$ and the function $\Phi$ 
appearing in the specification of our algorithms were fixed 
to a priori minimal values based on probabilistic arguments 
on the distribution of primes, then increased when needed.

\subsection{Prime norm ideals}
\label{sec:primeres}

We show experimental results on the prime norm algorithm of 
Section~\ref{sec:ideals:prime} in Figure~\ref{fig:primeres}. 
The norms of the ideals constructed seem to be slightly 
larger than $p^{1/2}$ and the computation time cubic in 
$\log(p)$. 

\begin{figure}
\begin{center}
\includegraphics[clip=true,viewport=4cm 9cm 19cm 20.5cm,scale=0.4]{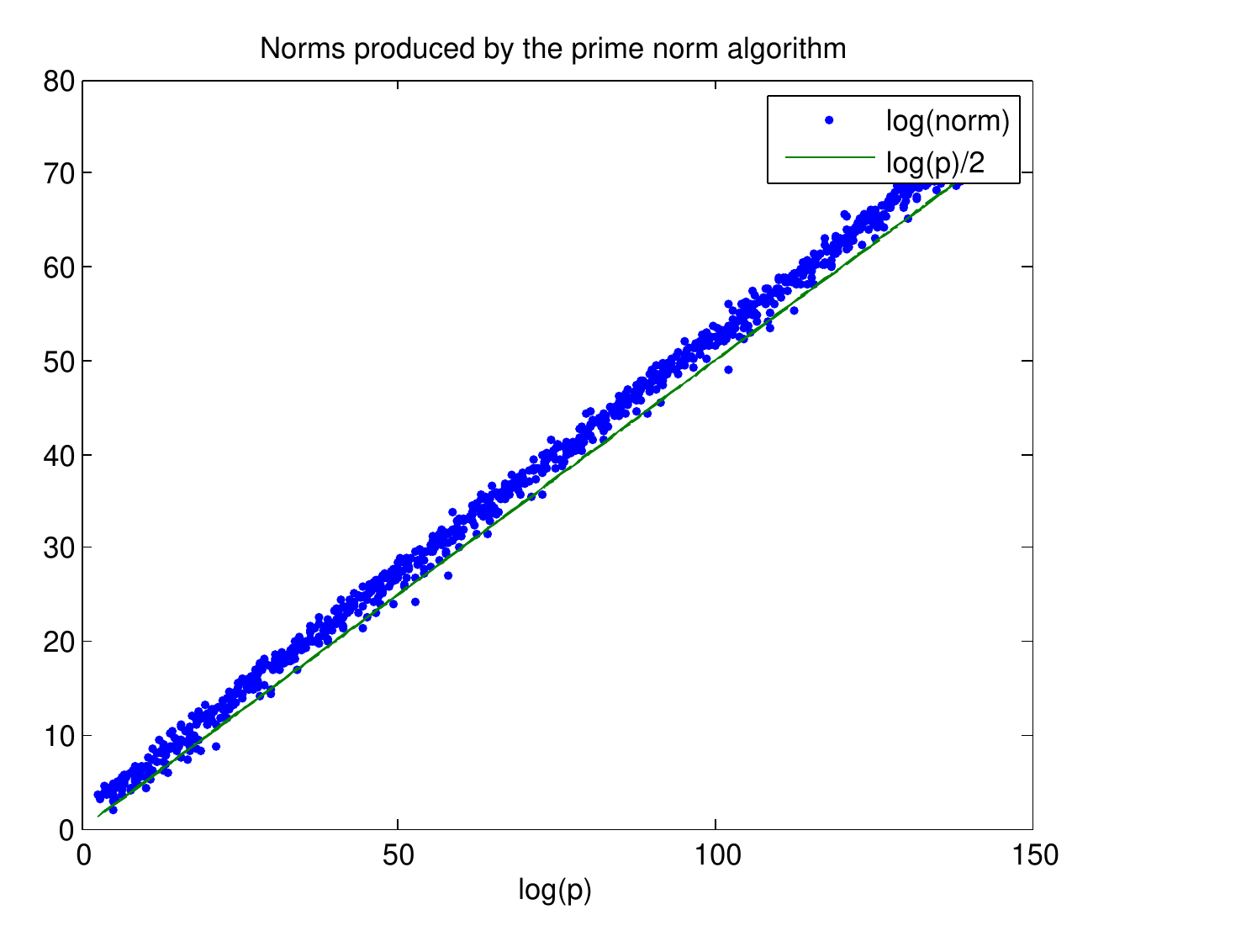}
\includegraphics[clip=true,viewport=3.8cm 9cm 19cm 20.5cm,scale=0.4]{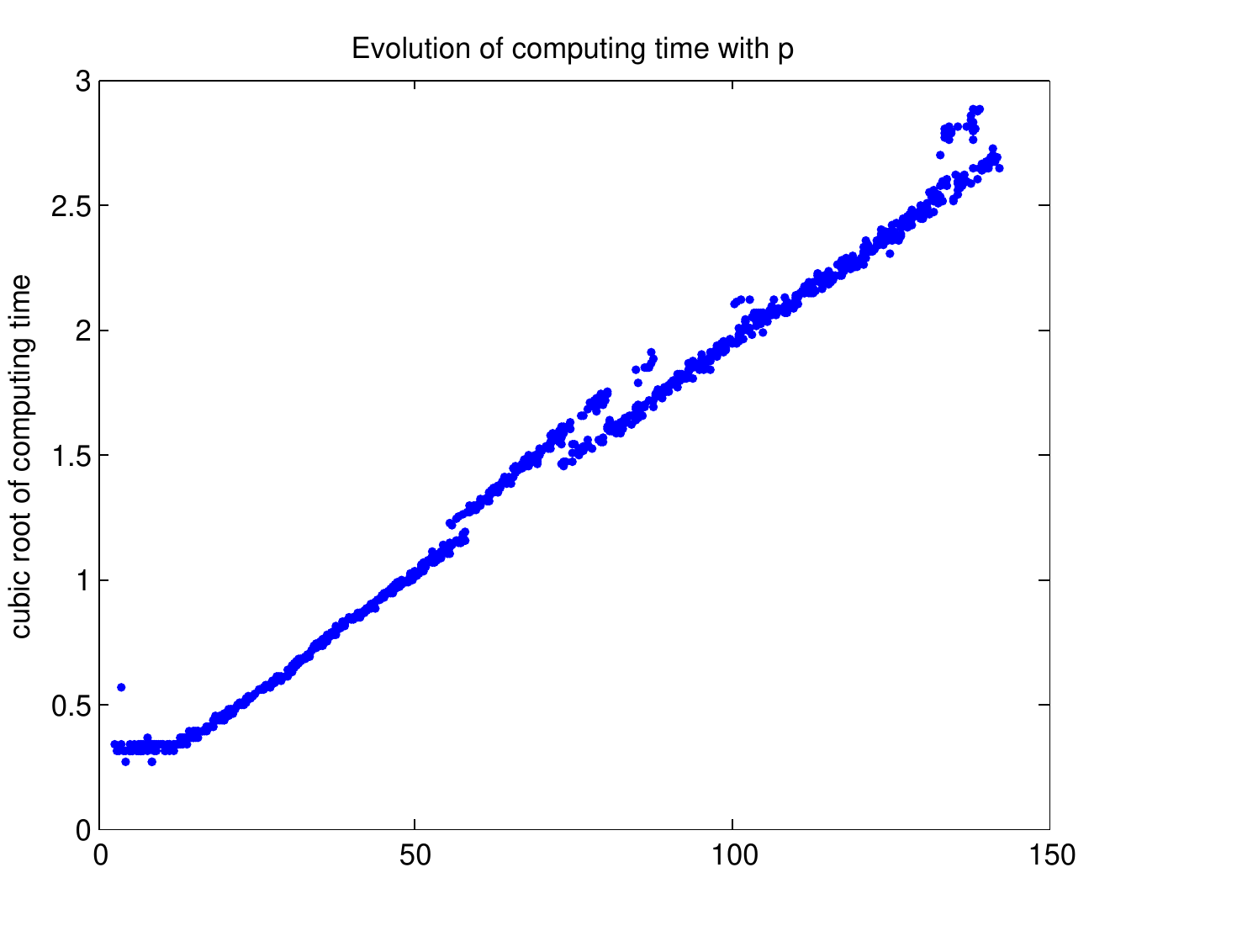}
\end{center}
\caption{Experimental results for the algorithm of Section~\ref{sec:ideals:prime} (with $m$ of the expected size): logarithm of the output norm $q_I(\alpha)$ and cubic root of running time with respect to $\log p$.
\label{fig:primeres}}
\end{figure}

\subsection{Quaternion elements with particular norms}
\label{sec:resNorm}

Experimental results on the algorithm of Section~\ref{sec:repns_in_orders} are shown in Figures~\ref{fig:Nellpowel} and~\ref{fig:ellpowel}, respectively for computing elements of norms $\ell^e$ or $N\ell^e$, for some~$e$. 
The results show the difference between the minimal exponent $e$ needed and a prediction based on probabilitic arguments.
All computations took less than one second.

\begin{figure}
\begin{center}
\includegraphics[clip=true,viewport=4cm 9cm 19cm 20.5cm,scale=0.4]{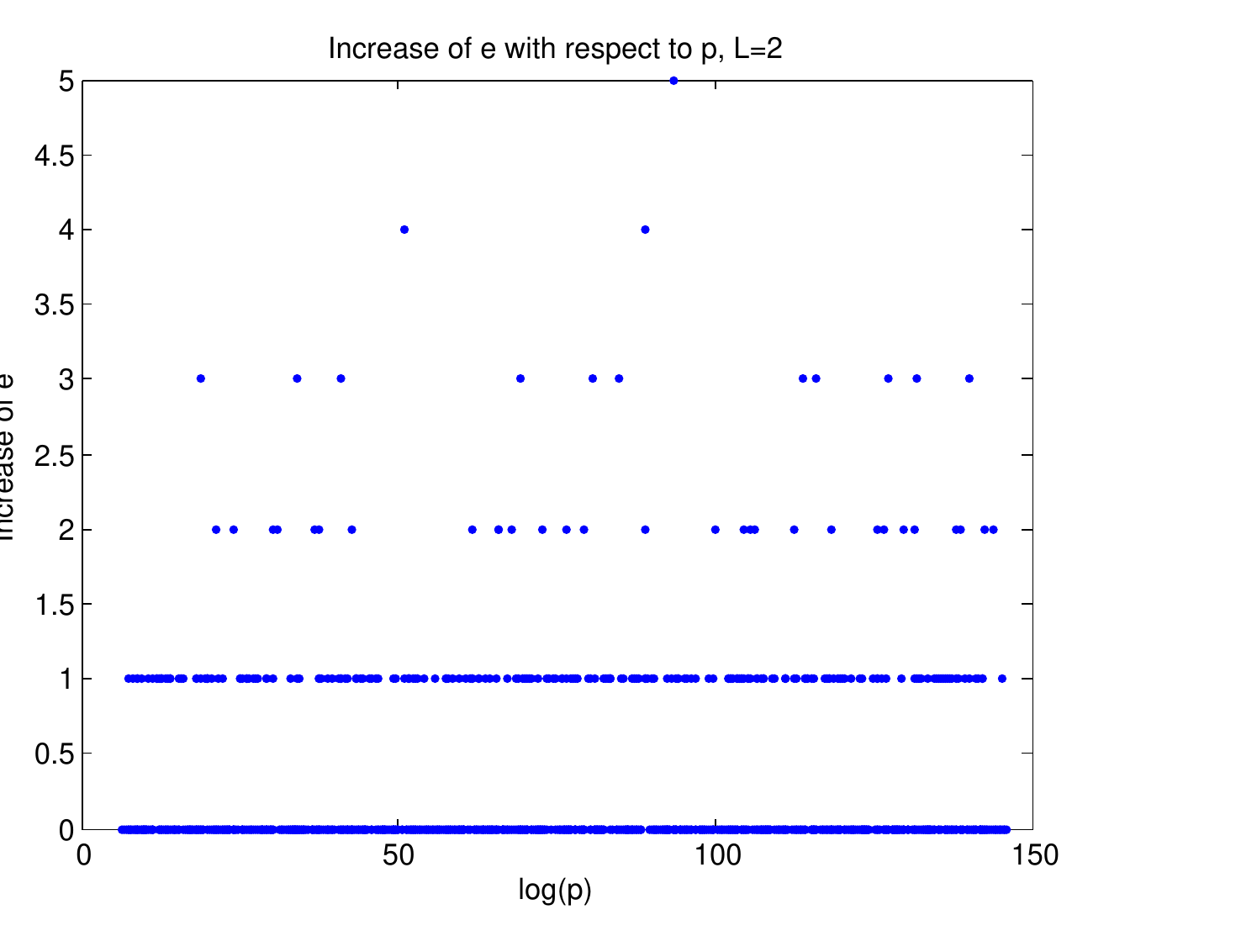}
\includegraphics[clip=true,viewport=4cm 9cm 19cm 20.5cm,scale=0.4]{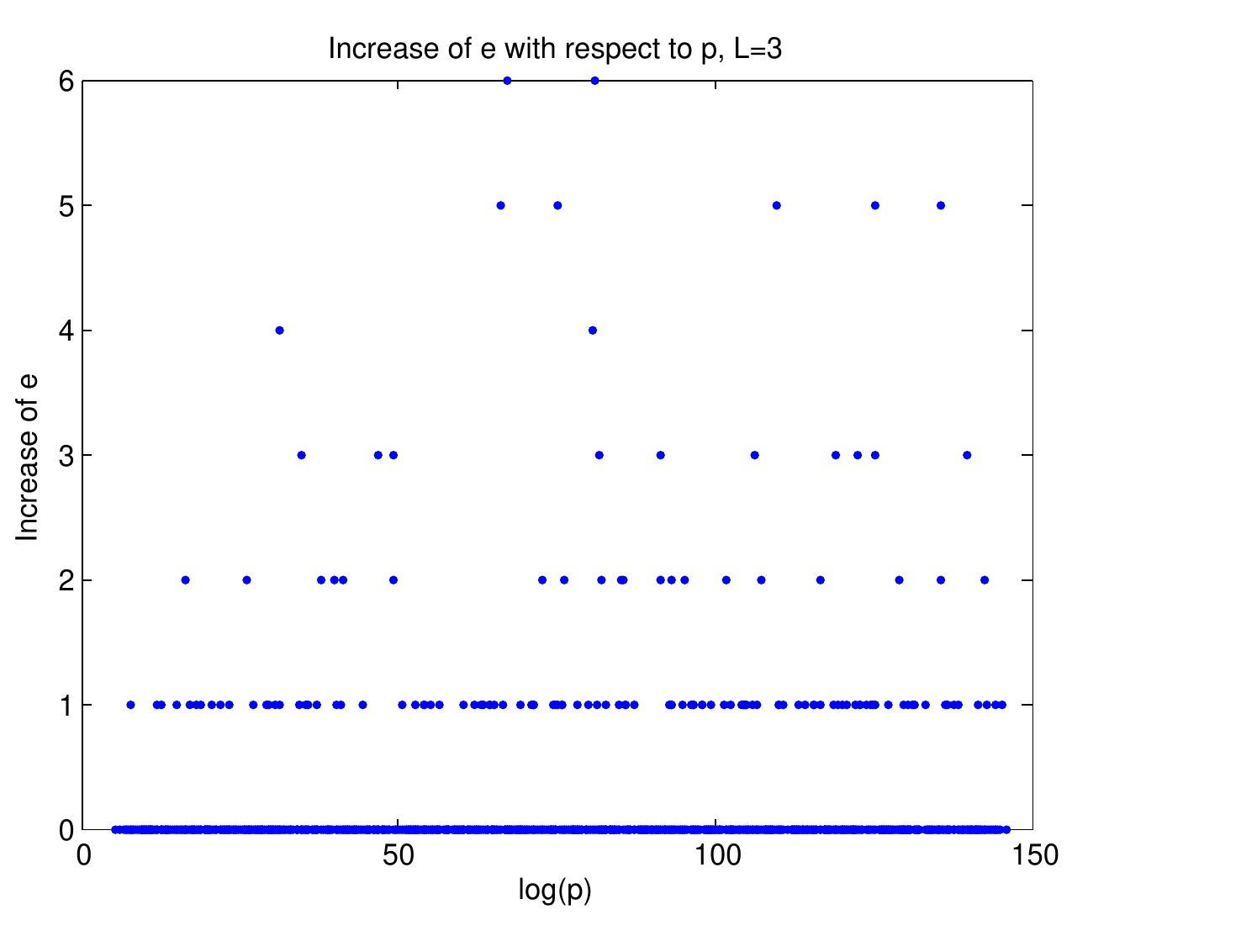}
\end{center}
\caption{
Experimental results for computing elements of norms $N\ell^e$ with the algorithm of Section~\ref{sec:repns_in_orders}, for various $p$ values with $\ell=2$ (left) and $\ell=3$ (right):  Difference between the minimal exponent $e$ needed and a prediction based on probabilistic arguments.
\label{fig:Nellpowel}}
\end{figure}

\begin{figure}
\begin{center}
\includegraphics[clip=true,viewport=4cm 9cm 19cm 20.5cm,scale=0.4]{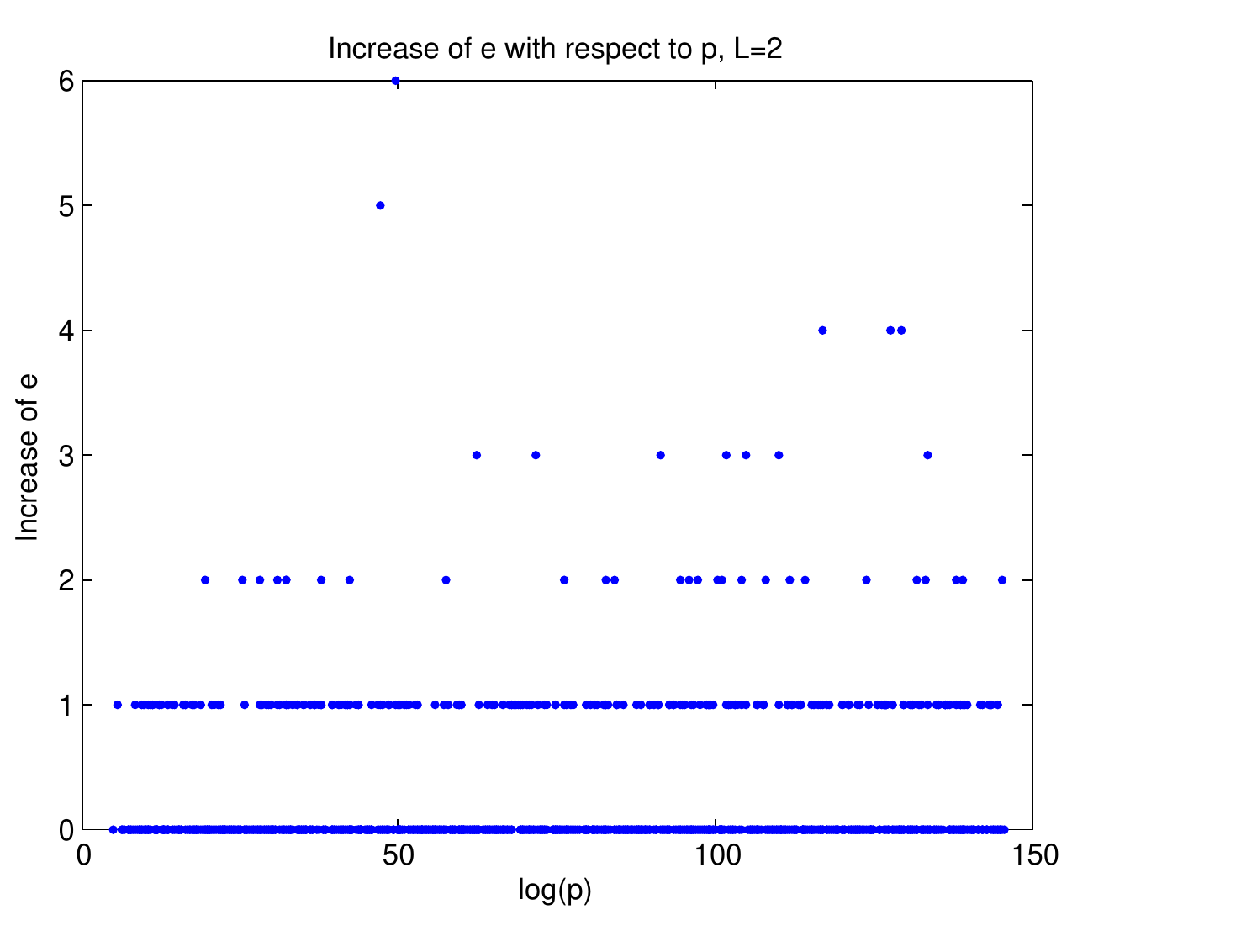}
\includegraphics[clip=true,viewport=4cm 9cm 19cm 20.5cm,scale=0.4]{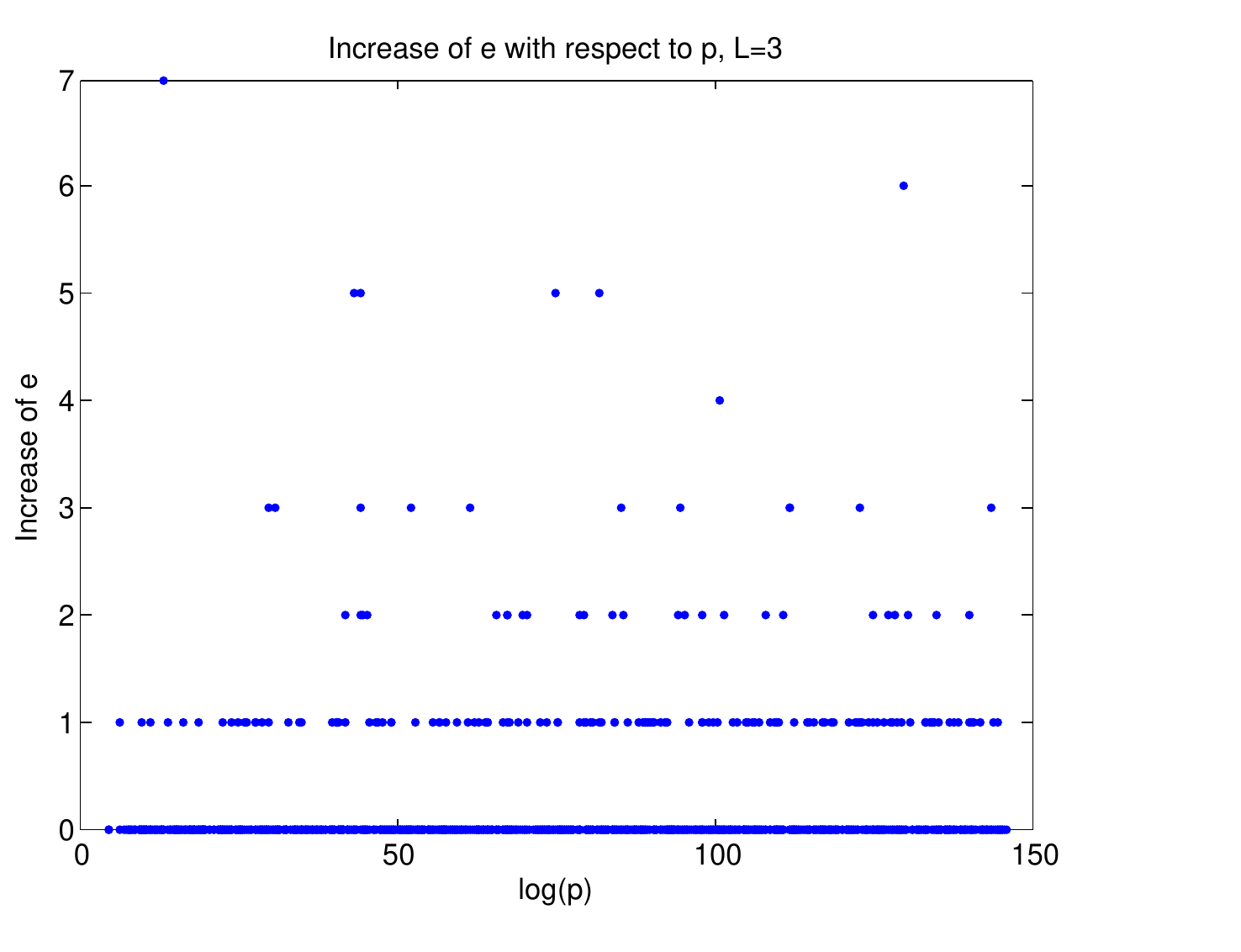}
\includegraphics[clip=true,viewport=4cm 9cm 19cm 20.5cm,scale=0.4]{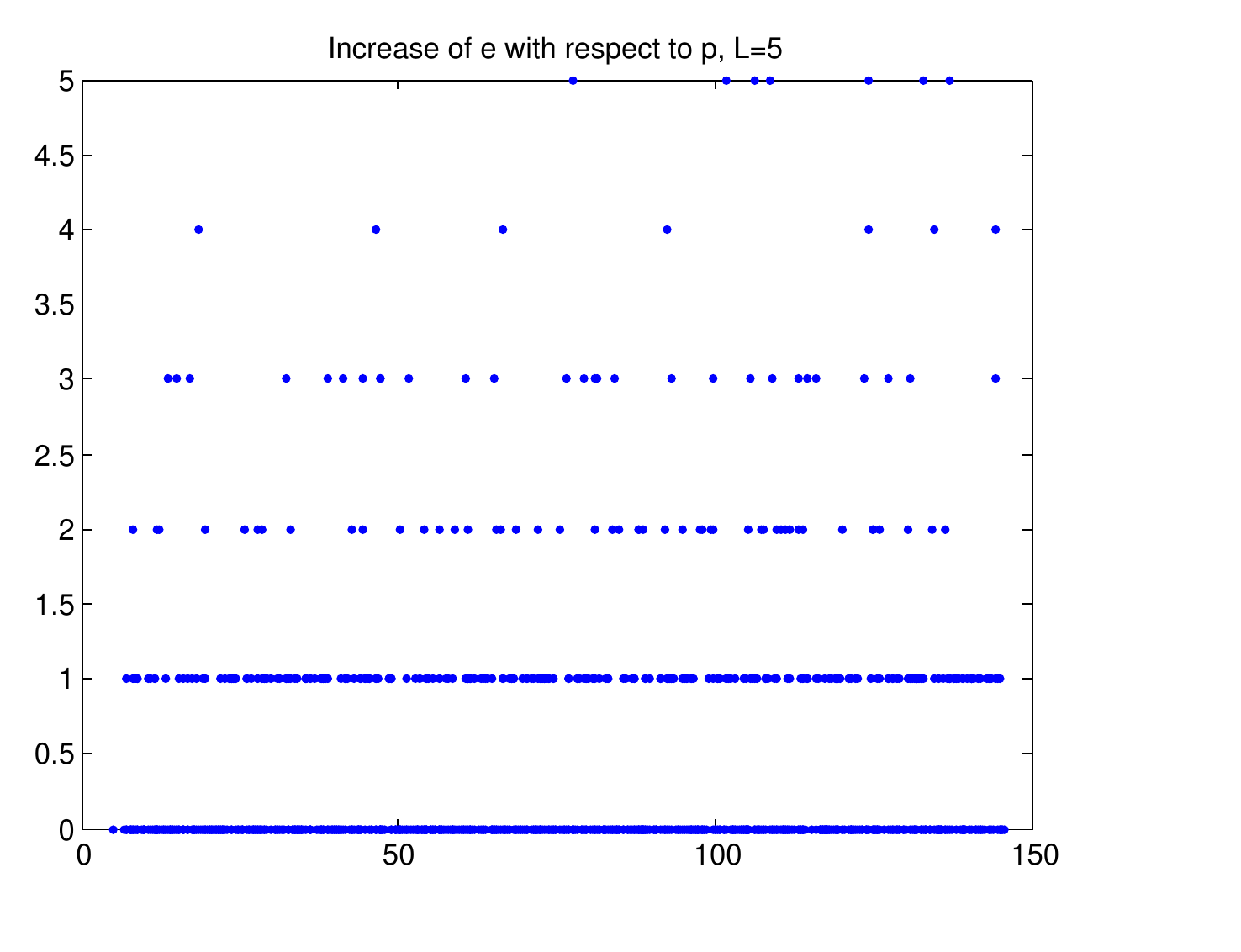}
\includegraphics[clip=true,viewport=4cm 9cm 19cm 20.5cm,scale=0.4]{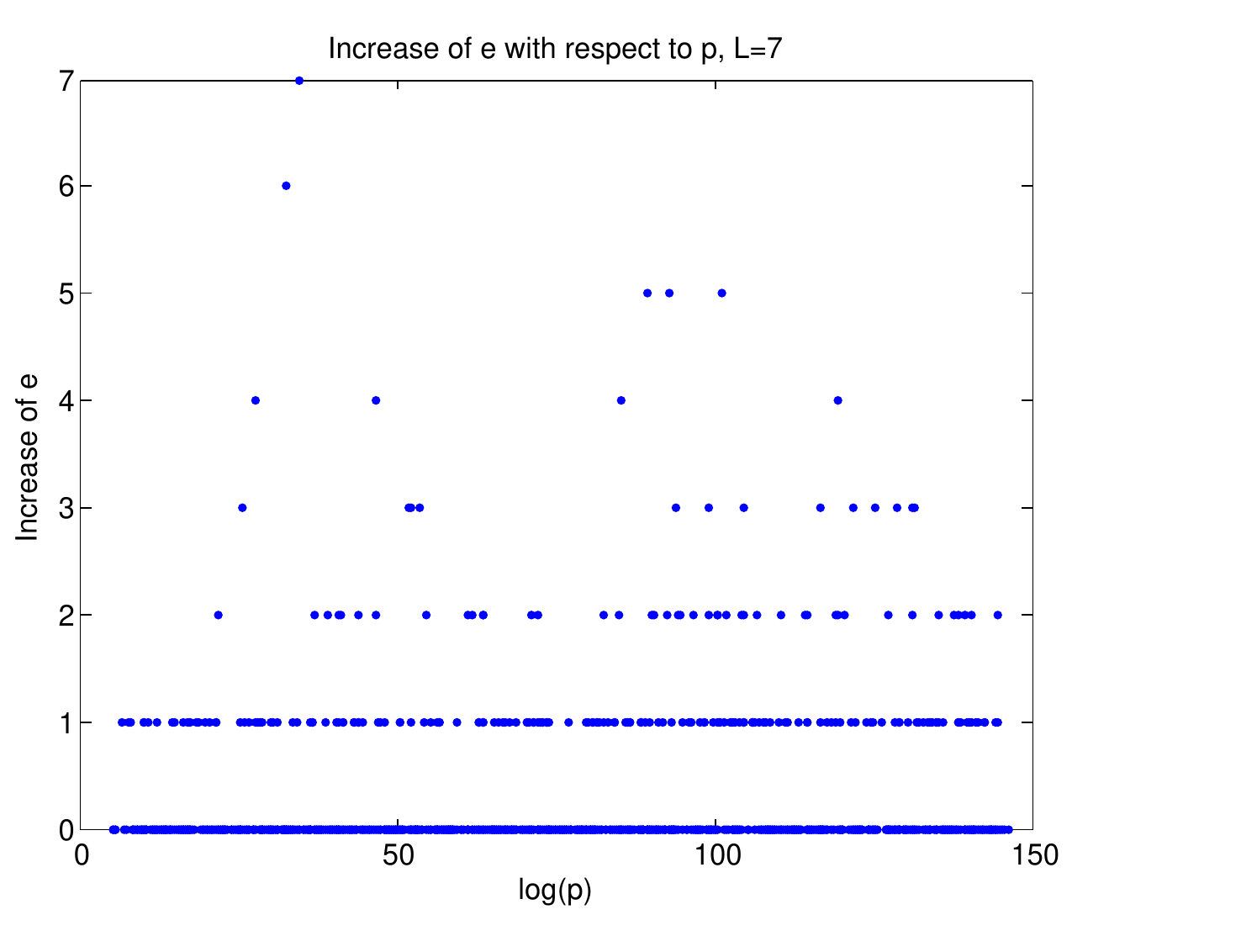}
\end{center}
\caption{
Experimental results for computing elements of norms $\ell^e$ with the algorithm of Section~\ref{sec:repns_in_orders}, for $\ell\in\{2,3,5,7\}$ and various $p$ value:  Difference between the minimal exponent $e$ needed and a prediction based on probabilistic arguments.
\label{fig:ellpowel}}
\end{figure}

\subsection{Ideals with $\ell$-power norms}

Experimental results on the algorithms of Section~\ref{sec:ideals:ellpow} are shown in Figures~\ref{fig:liftSize}, \ref{fig:liftTime}, ~\ref{fig:ellpowSize}, \ref{fig:ellpowTime}.

\begin{figure}
\begin{center}
\includegraphics[clip=true,viewport=3.8cm 9cm 19cm 20.5cm,scale=0.4]{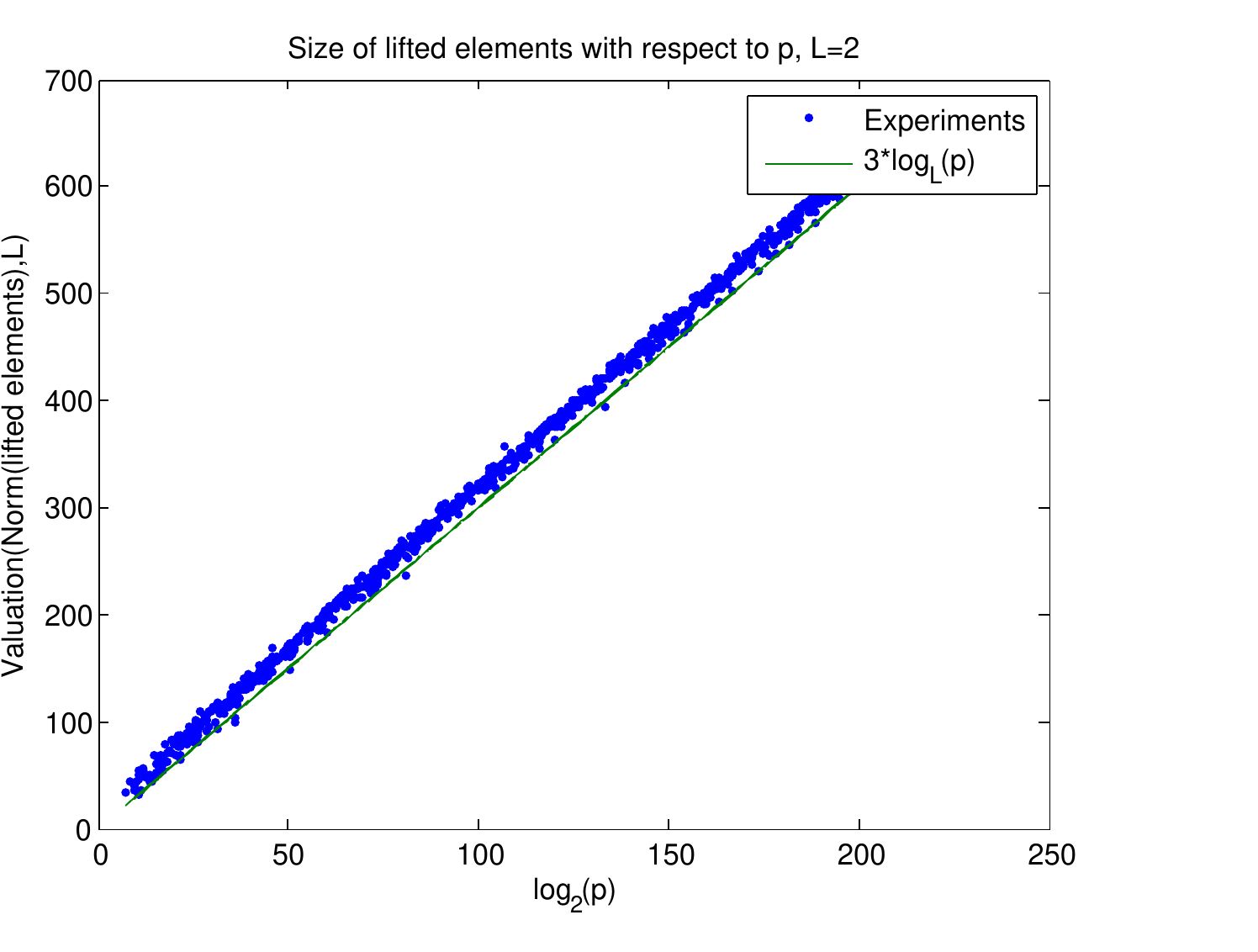}
\includegraphics[clip=true,viewport=3.8cm 9cm 19cm 20.5cm,scale=0.4]{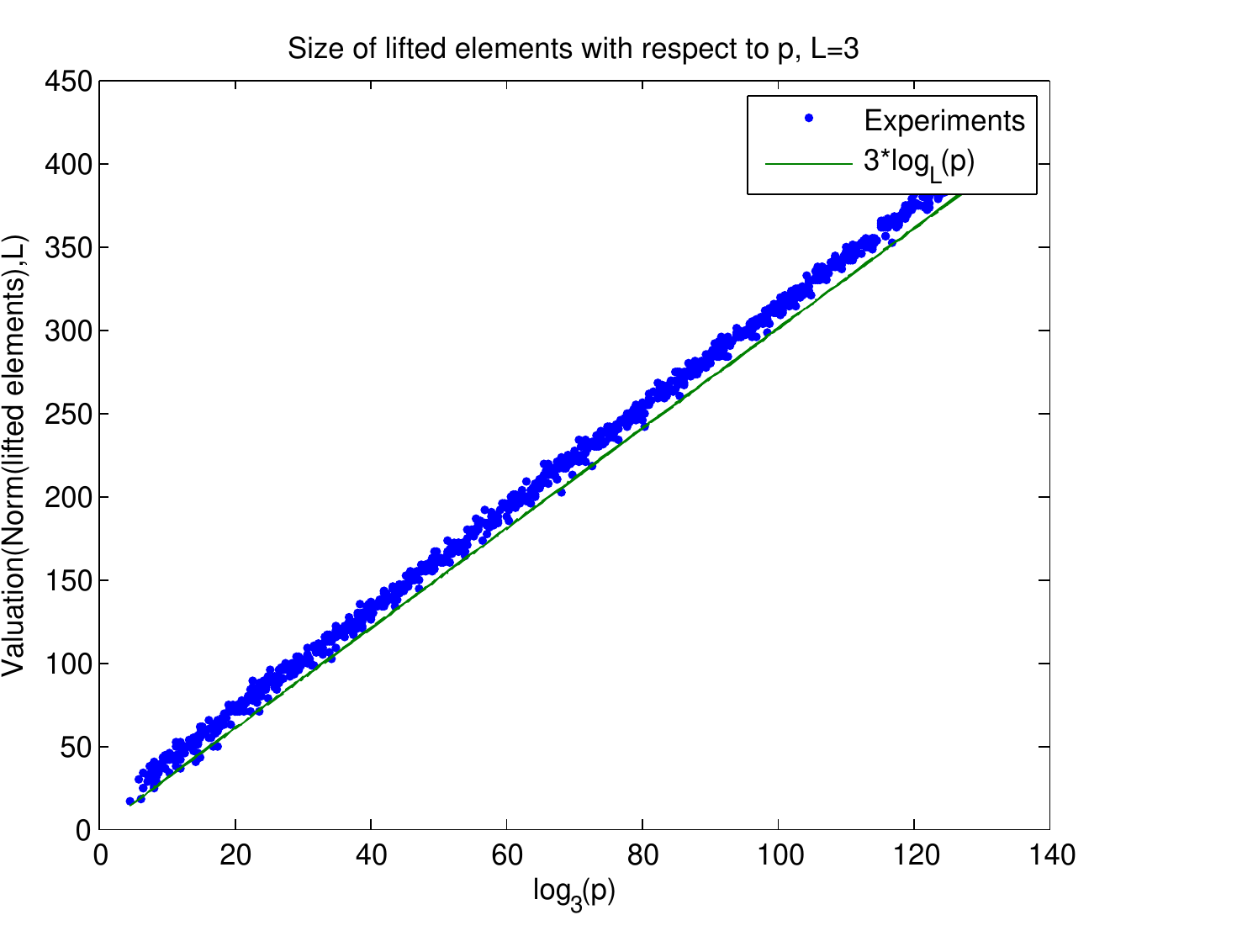}
\end{center}
\caption{Size of $\ell$-power norm quaternions obtained with the algorithm of Section~\ref{sec:ideals:ellpow:lift} for various $p$ values with $\ell=2$ (left) and $\ell=3$ (right). The green line corresponds to the approximated values $3\log_\ell p$.
\label{fig:liftSize}}
\end{figure}

\begin{figure}
\begin{center}
\includegraphics[clip=true,viewport=3.8cm 9cm 19cm 20.5cm,scale=0.4]{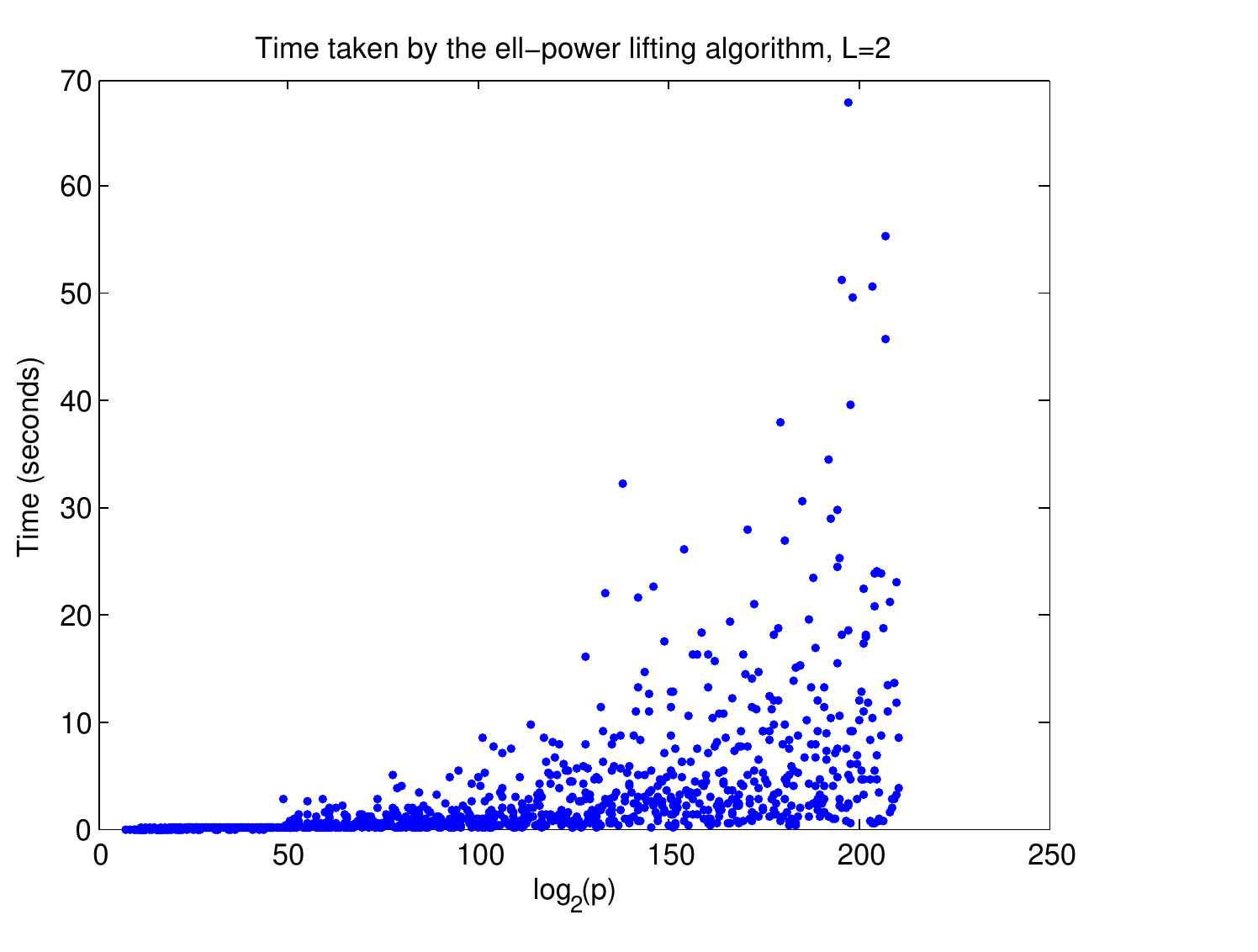}
\includegraphics[clip=true,viewport=3.8cm 9cm 19cm 20.5cm,scale=0.4]{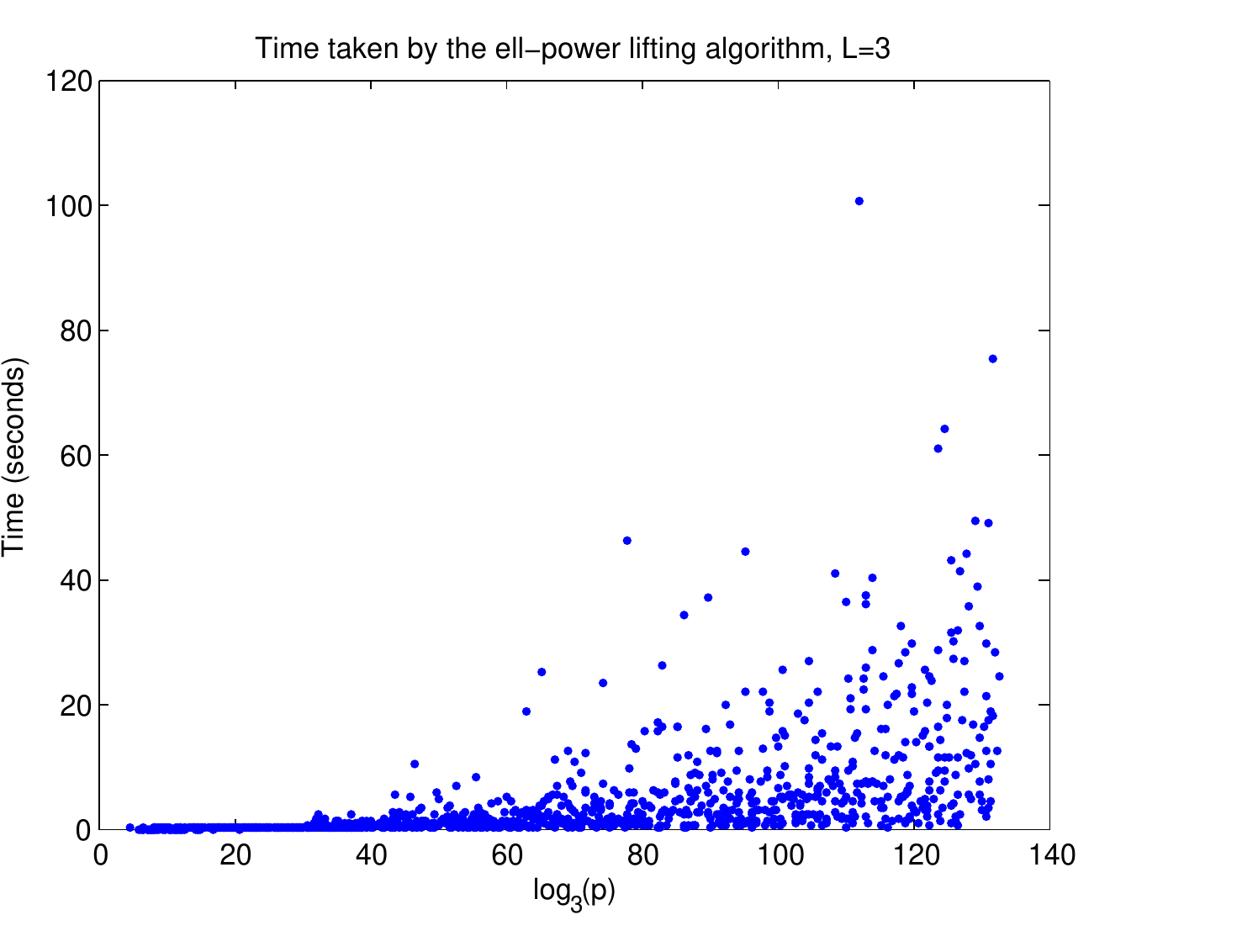}
\end{center}
\caption{Time taken by the algorithm of Section~\ref{sec:ideals:ellpow:lift} 
for various $p$ values, with~$\ell=2$~(left) and $\ell=3$~(right).
\label{fig:liftTime}}
\end{figure}

\begin{figure}
\begin{center}
\includegraphics[clip=true,viewport=3.8cm 9cm 19cm 20.5cm,scale=0.4]{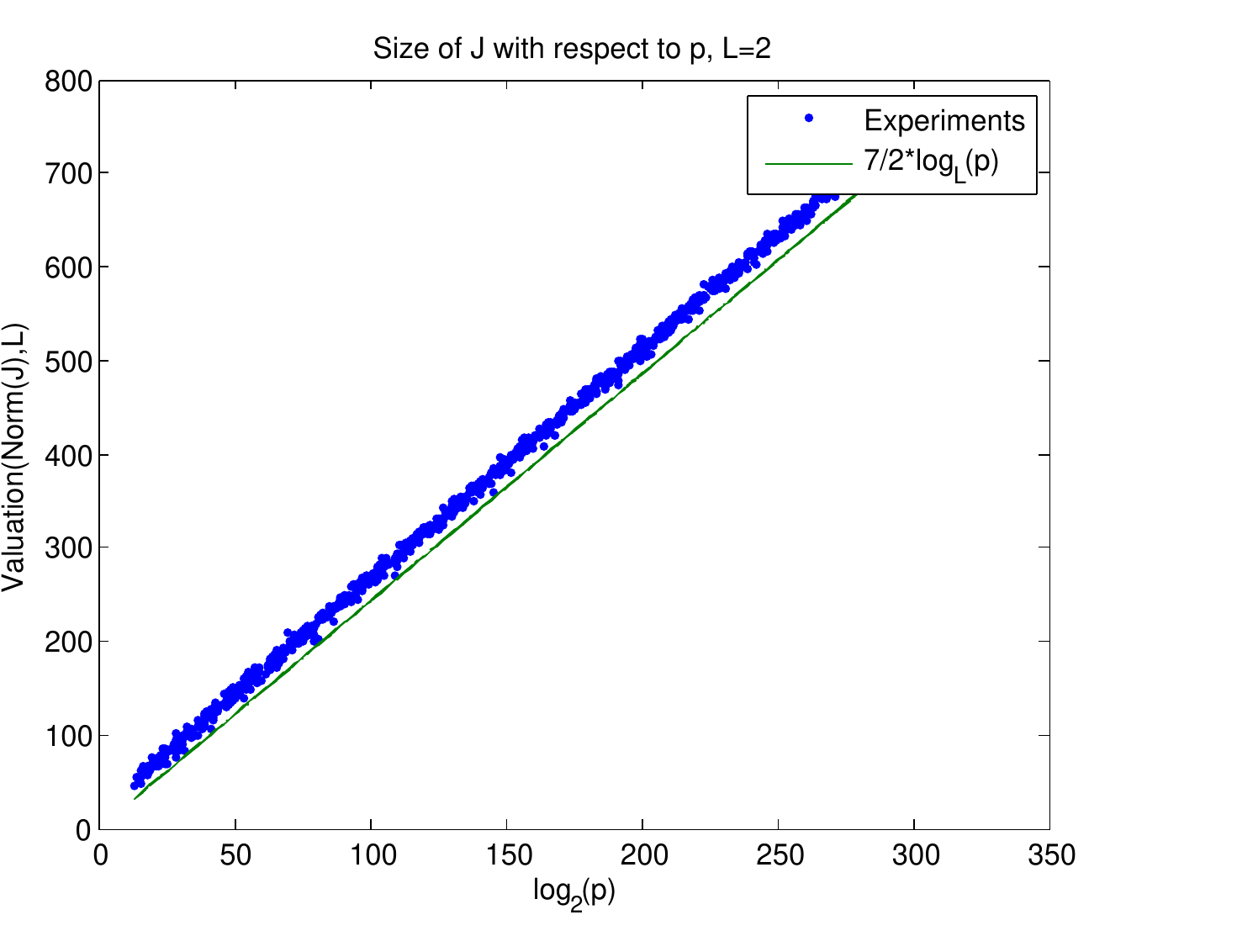}
\includegraphics[clip=true,viewport=3.8cm 9cm 19cm 20.5cm,scale=0.4]{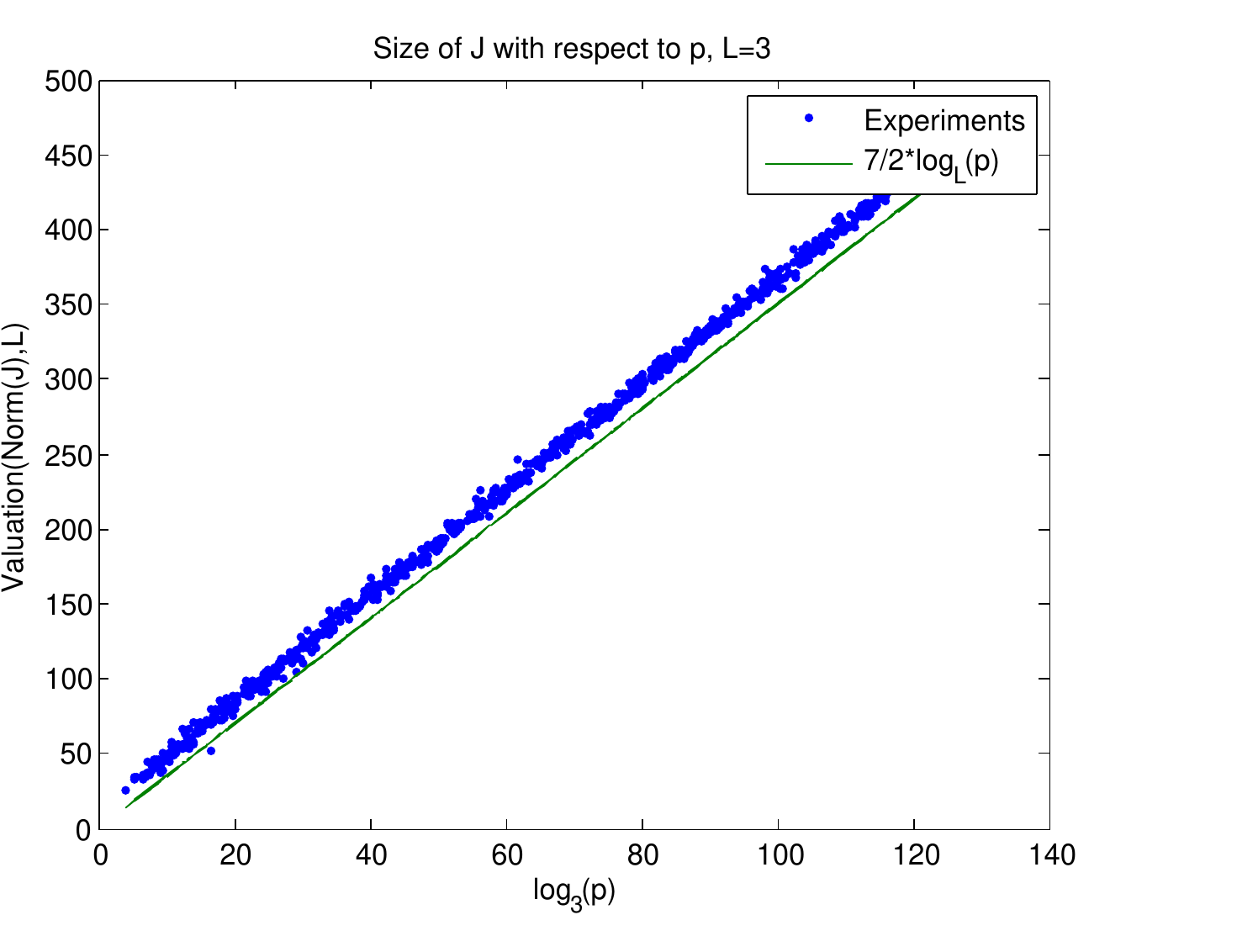}
\end{center}
\caption{Size of $\ell$-power norm ideals returned by the algorithm 
of Section~\ref{sec:ideals:ellpow} for various $p$ values with $\ell=2$ 
(left) and $\ell=3$ (right). The green line shows a priori approximative 
values $\frac{7}{2}\log_\ell p$.
\label{fig:ellpowSize}}
\end{figure}

\begin{figure}
\begin{center}
\includegraphics[clip=true,viewport=3.8cm 9cm 19cm 20.5cm,scale=0.4]{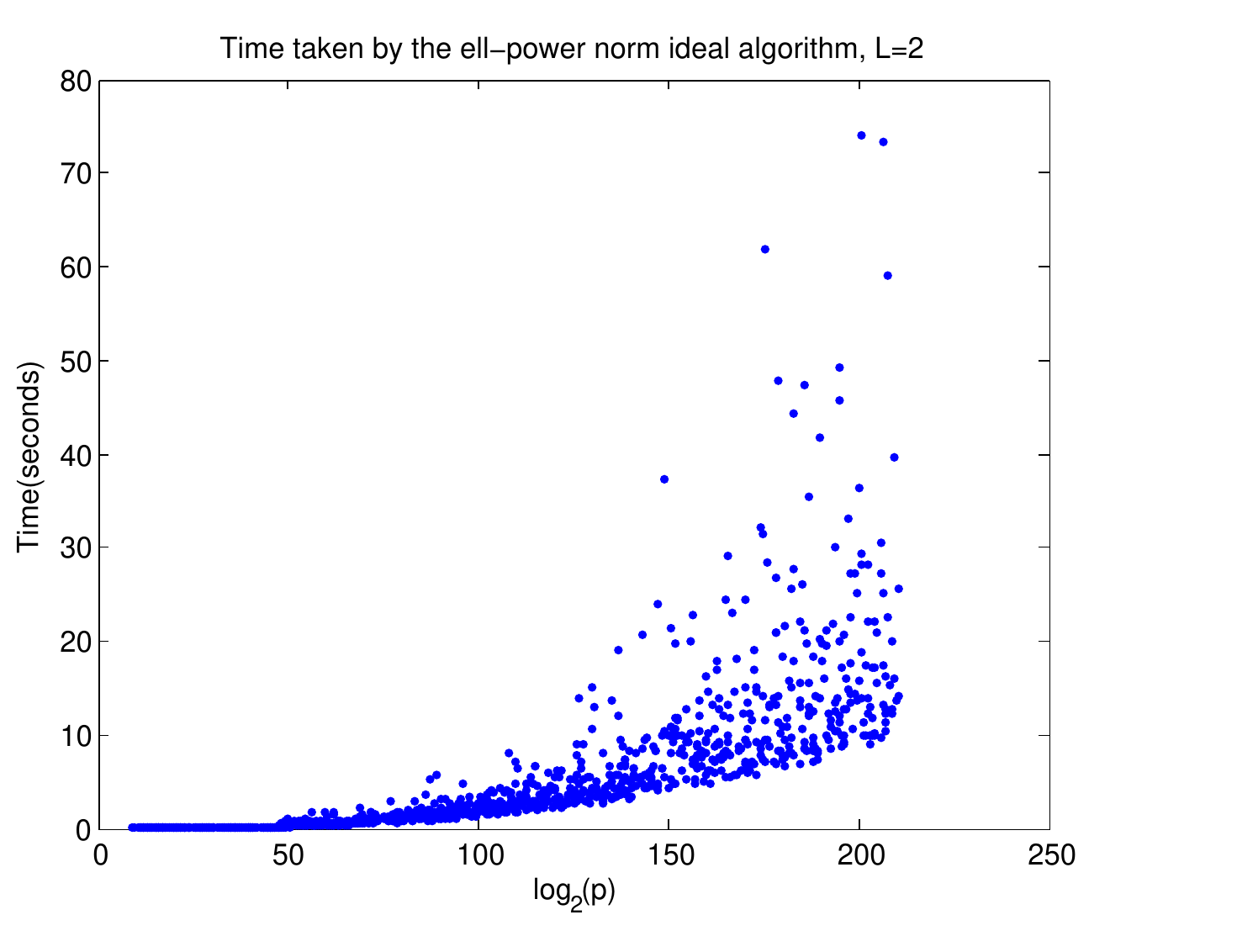}
\includegraphics[clip=true,viewport=3.8cm 9cm 19cm 20.5cm,scale=0.4]{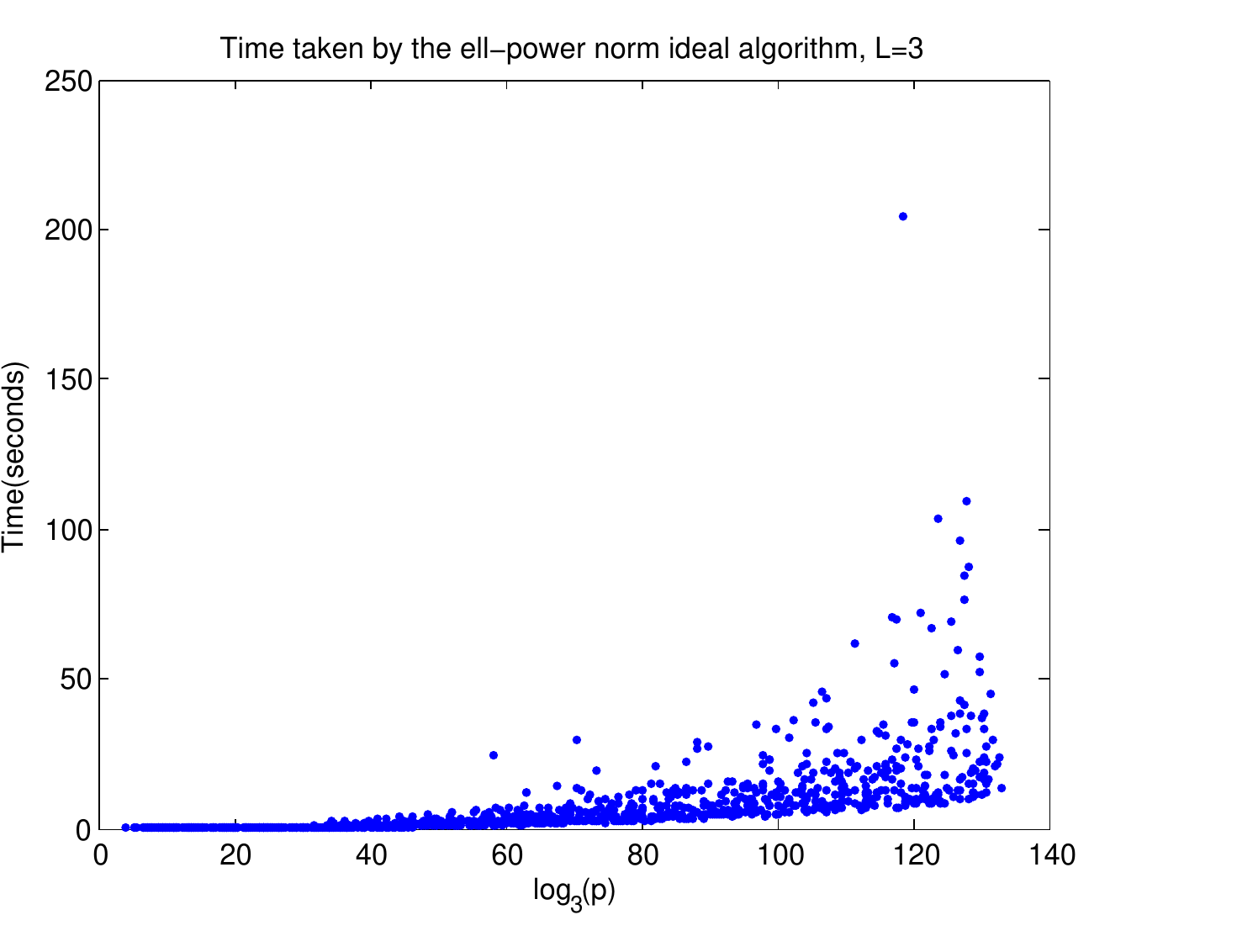}
\end{center}
\caption{Time taken by the algorithm of Section~\ref{sec:ideals:ellpow} 
for various $p$ values with $\ell=2$ (left) and $\ell=3$ (right).
\label{fig:ellpowTime}}
\end{figure}


\begin{thebibliography}{10}

\bibitem{Ankeny1952} N.~C.~Ankeny.
The least quadratic non residue,
{\em Annals of Mathematics}, 55(1):65--72, 1952.

\bibitem{Bach1990}
E.~Bach.
\newblock Explicit bounds for primality testing and related problems,
\newblock {\em Math. Comp.}, 55(191):355--380, 1990.

\bibitem{Cassels1967} J.~W.~S. Cassels.
Global fields.
In J.~W.~S. Cassels and A.~Frohlich, editors, 
{\em Algebraic Number Theory}, 
chapter Global Fields, pages 42--84. Academic Press, 1967.

\bibitem{Charles2009} D.~X.~Charles, K.~E.~Lauter, and E.~Z.~Goren.
\newblock Cryptographic hash functions from expander graphs.
\newblock {\em J. Cryptology}, 22(1):93--113, 2009.

\bibitem{Cornacchia1903} G.~Cornacchia.
Su di un metodo per la risoluzione in numeri interi 
dell' equazione $\sum_{h=0}^nc_hx^{n-h}y^h=p$,
{\em Giornale di Matematiche di Battaglini}, 46:33--90, 1903.

\bibitem{Deuring1941} M.~Deuring.
Die {T}ypen der {M}ultiplikatorenringe elliptischer {F}unktionenk\"orper.
{\em Abhandlungen aus dem Mathematischen Seminar der Universit\"at Hamburg}, 
14:197--272, 1941.

\bibitem{Magma} W.~Bosma, J.~J.~Cannon, C.~Fieker, A.~Steel (eds.), 
Handbook of Magma functions, Edition 2.20 (2013), 
\url{http://http://magma.maths.usyd.edu.au/magma/}.

\bibitem{Heath-Brown1989} D.~R.~Heath-Brown.
The number of primes in a short interval.
{\em J. Reine Angew.~Math.}, 397:162--193, 1989.

\bibitem{Kohel1996} D.~Kohel.
{\em Endomorphism rings of elliptic curves over finite fields},
PhD thesis, University of California, Berkeley, 1996.

\bibitem{Maier1985} H.~Maier.
Primes in short intervals,
{\em Michigan Math.~J.}, 32:221--225, 1985.

\bibitem{Petit2008c}
C.~Petit, K.~Lauter, and J.-J.~Quisquater.
Full cryptanalysis of {LPS} and {Morgenstern} hash functions.
In R.~Ostrovsky, R.~De Prisco, and I.~Visconti, eds.,
{\em SCN}, volume 5229 of {\em Lecture Notes in Computer Science}, 
pages 263--277. Springer, 2008.

\bibitem{Pizer1980} A.~Pizer.
An algorithm for computing modular forms on {$\Gamma_0(N)^*$}.
{\em Journal of Algebra}, 64:340--390, 1980.

\bibitem{Selberg1943} A.~Selberg.
On the normal density of primes in small intervals 
and the difference between consecutive primes, 
{\em Arch. Math. Naturvid.}, 47:87--105, 1943.

\bibitem{Vigneras1980} M.-F.~Vign\'eras.
{\em Arithm\'etique des alg\`ebres de quaternions}.
Springer-Verlag, 1980.

\end{thebibliography}
\end{document}